\documentclass[11pt]{amsart}
\usepackage{amsmath,amssymb}
\usepackage[dvips]{graphicx}
\usepackage{psfrag}
\usepackage{enumitem}
\usepackage{hyperref}
\title[Proof of the $1$-factorization \& Hamilton decomposition conjectures IV]{Proof of the $1$-factorization and Hamilton
decomposition conjectures IV: exceptional systems for the two cliques case}
\author{Daniela K\"uhn, Allan Lo and Deryk Osthus}
\thanks {The research leading to these results was partially supported by the  European Research Council
under the European Union's Seventh Framework Programme (FP/2007--2013) / ERC Grant
Agreement n. 258345 (D.~K\"uhn and A.~Lo) and 306349 (D.~Osthus).
The research was also partially supported by the EPSRC, grant no. EP/J008087/1 (D.~K\"uhn and D.~Osthus).
}
\date{\today}

\def\COMMENT#1{}
\def\TASK#1{}
\begin{document}
\numberwithin{equation}{section}
\def\noproof{{\unskip\nobreak\hfill\penalty50\hskip2em\hbox{}\nobreak\hfill%
        $\square$\parfillskip=0pt\finalhyphendemerits=0\par}\goodbreak}
\def\endproof{\noproof\bigskip}
\newdimen\margin   % needed for macros \textdisplay & \ltextdisplay
\def\textno#1&#2\par{%
    \margin=\hsize
    \advance\margin by -4\parindent
           \setbox1=\hbox{\sl#1}%
    \ifdim\wd1 < \margin
       $$\box1\eqno#2$$%
    \else
       \bigbreak
       \hbox to \hsize{\indent$\vcenter{\advance\hsize by -3\parindent
       \sl\noindent#1}\hfil#2$}%
       \bigbreak
    \fi}
\def\proof{\removelastskip\penalty55\medskip\noindent{\bf Proof. }}

\def\C{\mathcal{C}}

\def\prob{\mathbb{P}}

\def\dgreen{\lambda_1}
\def\dblue{\lambda_{\rm EF}}
\def\epszero{\eps_0}
\def\d0{d_0}
\def\epsdecom{\widetilde{\epsilon}_{1}}
\def\eps{{\varepsilon}}
\newcommand{\ex}{\mathbb{E}}
\newcommand{\pr}{\mathbb{P}}
\newcommand{\cB}{\mathcal{B}}
\newcommand{\cS}{\mathcal{S}}
\newcommand{\cF}{\mathcal{F}}
\newcommand{\cC}{\mathcal{C}}
\newcommand{\cP}{\mathcal{P}}
\newcommand{\cQ}{\mathcal{Q}}
\newcommand{\cR}{\mathcal{R}}
\newcommand{\cK}{\mathcal{K}}
\newcommand{\cD}{\mathcal{D}}
\newcommand{\cI}{\mathcal{I}}
\newcommand{\cV}{\mathcal{V}}
\newcommand{\cT}{\mathcal{T}}
\newcommand{\eul}{{\rm e}}

\newcommand{\cJ}{\mathcal{J}}

\newtheorem{firstthm}{Proposition}[section]
\newtheorem{thm}[firstthm]{Theorem}
\newtheorem{prop}[firstthm]{Proposition}
\newtheorem{fact}[firstthm]{Fact}
\newtheorem{lemma}[firstthm]{Lemma}
\newtheorem{cor}[firstthm]{Corollary}
\newtheorem{problem}[firstthm]{Problem}
\newtheorem{defin}[firstthm]{Definition}
\newtheorem{conj}[firstthm]{Conjecture}
\newtheorem{conjecture}[firstthm]{Conjecture}
\newtheorem{claim}[firstthm]{Claim}
\newtheorem{remark}[firstthm]{Remark}

\begin{abstract}
In a sequence of four papers, we prove the following results (via a unified approach) for all sufficiently large $n$:
\begin{itemize}
\item[(i)] [\emph{$1$-factorization conjecture}]
Suppose that $n$ is even and $D\geq 2\lceil n/4\rceil -1$. 
Then every $D$-regular graph $G$ on $n$ vertices has a decomposition into perfect matchings.
Equivalently, $\chi'(G)=D$.

\item[(ii)] [\emph{Hamilton decomposition conjecture}]
Suppose that $D \ge   \lfloor n/2 \rfloor $.
Then every $D$-regular graph $G$ on $n$ vertices has a decomposition
into Hamilton cycles and at most one perfect matching.

\item[(iii)] We prove an optimal result on the number of edge-disjoint Hamilton cycles in a graph 
of given minimum degree.
\end{itemize}
According to Dirac, (i) was first raised in the 1950s.
(ii) and (iii)
%and the special case $\delta= \lceil n/2 \rceil$ of (iii)
answer questions of Nash-Williams from 1970.
The above bounds are best possible. In the current paper, we prove results on the decomposition 
of sparse graphs into path systems. These are used in the proof of (i) and (ii)
in the case when $G$ is close to the union of two disjoint cliques.
\end{abstract}

\maketitle

\section{Introduction}

\subsection{Background and results}

In a sequence of four papers, we develop a unified approach to prove the following results on Hamilton decompositions and 
$1$-factorizations. 
The first of these results confirms the so-called $1$-factorization conjecture for all sufficiently large graphs.
(A \emph{$1$-factorization} of a graph~$G$ consists of a set of edge-disjoint perfect matchings covering all edges of~$G$.)
This conjecture was first stated explicitly by Chetwynd and Hilton~\cite{1factorization,CH}.
However, they wrote that according to Dirac, it was already discussed in the 1950s.%
\COMMENT{AL: 1950s same with abstract.} 

\begin{thm}\label{1factthm}
There exists an $n_0 \in \mathbb N$ such that the following holds.
Let $ n,D \in \mathbb N$ be such that $n\geq n_0$ is even and $D\geq 2\lceil n/4\rceil -1$. 
Then every $D$-regular graph $G$ on $n$ vertices has a $1$-factorization.%
    \COMMENT{So this means that $D\ge n/2$ if $n = 2 \pmod 4$ and $D\ge n/2-1$ if $n = 0 \pmod 4$.}
    Equivalently, $\chi'(G)=D$.
\end{thm}
The bound on the degree in Theorem~\ref{1factthm} is best possible.
Nash-Williams~\cite{initconj,decompconj} raised the related problem of finding a Hamilton decomposition 
in an even-regular graph. Here
a decomposition  of an (even-regular) graph~$G$ into Hamilton cycles consists of a set of edge-disjoint Hamilton cycles covering all edges of~$G$.
If $G$ is a regular graph of odd degree, it is natural to ask for a perfect matching in $G$ together with a decomposition of the remaining edges into Hamilton cycles.
\begin{thm} \label{HCDthm} 
There exists an $n_0 \in \mathbb N$ such that the following holds.
Let $ n,D \in \mathbb N$ be such that $n \geq n_0$ and
$D \ge   \lfloor n/2 \rfloor $.
Then every $D$-regular graph $G$ on $n$ vertices has a decomposition into Hamilton cycles and 
at most one perfect matching.
\end{thm}
Again, the bound on the degree in Theorem~\ref{HCDthm} is best possible and so the theorem confirms the conjecture of 
Nash-Williams for all sufficiently large graphs.

Finally (in combination with~\cite{KLOmindeg}), we also prove an optimal result on the number of edge-disjoint Hamilton cycles one can guarantee
in a graph of given minimum degree, which (as a special case) answers another question of Nash-Williams.
For a detailed discussion of the results and their background we refer to~\cite{paper1}.

\subsection{Overall structure of the argument}

For all  of our main results, we split the argument according to the structure of the graph $G$ under consideration:
\begin{enumerate}
\item[{\rm (i)}] $G$ is  close to the complete balanced bipartite graph $K_{n/2,n/2}$;
\item[{\rm (ii)}] $G$ is close to the union of two disjoint copies of a clique $K_{n/2}$;
\item[{\rm (iii)}] $G$ is a `robust expander'.
\end{enumerate}
Roughly speaking, $G$ is a robust expander if for every set $S$ of vertices, the neighbourhood of $S$ is at least a little larger than $|S|$,
even if we delete a small proportion of the edges of $G$.
The main result of~\cite{Kelly} states that every dense regular robust expander
has a Hamilton decomposition.
This immediately implies Theorems~\ref{1factthm} and~\ref{HCDthm} in Case~(iii).

Case~(i) is proved in~\cite{paper2}. 
Most of the argument for Case~(ii) is contained in~\cite{paper1},
which also includes a more detailed discussion of the overall structure of the proof.
Some of the results needed for Case~(ii) (on decompositions into `exceptional path systems') are proved in the current paper.
Case~(ii) is by far the hardest case for Theorems~\ref{1factthm} and~\ref{HCDthm}, as the extremal examples are all close to the disjoint union of two 
cliques. The arguments in~\cite{paper2,paper1}
make use of an `approximate decomposition' result, which  is proved in~\cite{paper3}.

\subsection{Contribution of the current paper} \label{sketch}

As mentioned above, the current paper is concerned with Case~(ii), i.e.~when $G$ is close to the union of two cliques.
More precisely, we say that a graph $G$ on $n$ vertices is \emph{$\eps$-close to the union of two disjoint copies of $K_{n/2}$}
if there exists $A\subseteq V(G)$ with $|A|=\lfloor n/2\rfloor$ and such that $e(A,V(G)\setminus A)\le\eps n^{2}$.

We will prove results which are used in~\cite{paper1} to prove the following theorem, which 
is a common generalization of Theorems~\ref{1factthm} and Theorems~\ref{HCDthm} in Case~(ii).
Essentially, this theorem guarantees a decomposition into Hamilton cycles and perfect matchings which contains as many Hamilton cycles as possible.

\begin{thm}\label{1factstrong}
For every $\eps_{\rm ex} > 0$ there exists $n_0\in\mathbb{N}$ such that the following holds for all $n\ge n_0$.
Suppose that $D \ge n - 2 \lfloor n/4 \rfloor -1$ and that $G$ is a $D$-regular graph on $n$ vertices which
is $\eps_{\rm ex}$-close to the union of two disjoint copies of $K_{n/2}$.
Let $F$ be the size of a minimum cut in $G$.
Then $G$ can be decomposed into $\lfloor \min \{D,F\} /2 \rfloor$ Hamilton cycles and $D - 2 \lfloor \min \{D,F\} /2 \rfloor$ perfect matchings.
\end{thm}
When constructing the Hamilton cycles (and perfect matchings) guaranteed by Theorem~\ref{1factstrong}, a crucial step is to obtain a decomposition of the `exceptional edges'.
To define exceptional edges, we consider a suitable partition of $V(G)$ into sets $A,A_0,B,B_0$ so that 
$A$ and $B$ induce almost complete graphs on close to $n/2$ vertices and $A_0$, $B_0$ contain the (small number of) `exceptional vertices'
which have many neighbours in both $A':=A \cup A_0$ and $B':=B \cup B_0$.
The exceptional edges are all those edges incident to $A_0$ and $B_0$ as well as all those edges joining $A'$ to $B'$.
These exceptional edges will be decomposed into `exceptional (path) systems', and each such exceptional system
will be extended into a Hamilton cycle. (Actually, the exceptional systems  may contain some non-exceptional edges as well.)

The exceptional systems are constructed in the current paper. If we want to extend an exceptional system into a Hamilton cycle,
one obvious necessary property is that the exceptional system needs to contain two independent edges between $A'$ and $B'$.
Another requirement will be that these exceptional systems are `localized', i.e.~given a partition of $A$ and $B$ into clusters, each exceptional system uses only vertices from 
$A_0 \cup B_0$ as well as from one of the clusters in both $A$ and $B$.
Some further constraints are due to the overall structure of the argument, which we outline below.

In~\cite{paper3}, we show how one can extend a suitable set of exceptional systems to obtain an approximate decomposition of $G$,
i.e.~a set of edge-disjoint Hamilton cycles covering almost all edges of~$G$.
However, one does not have any control over the `leftover' graph~$H$, i.e.~it is not clear how to extend this into a decomposition.
In~\cite{Kelly} this problem was solved by introducing the concept of a `robustly decomposable graph'~$G^{\rm rob}$.
Roughly speaking, this is a sparse regular graph with the following property:
given \emph{any} very sparse regular graph~$H$ with $V(H)=V(G^{\rm rob})$ which is edge-disjoint from $G^{\rm rob}$,
one can guarantee that $G^{\rm rob} \cup H$ has a Hamilton decomposition.
This leads to a natural (and very general) strategy to obtain a decomposition of $G$:
\begin{itemize}
\item[(1)] find a (sparse) robustly decomposable graph~$G^{\rm rob}$ in $G$ and let $G'$ denote the leftover;
\item[(2)] find an approximate Hamilton decomposition of $G'$ and let $H$ denote the (very sparse) leftover;
\item[(3)] find a Hamilton decomposition of~$G^{\rm rob} \cup H$.
\end{itemize}
$G^{\rm rob}$ is constructed in~\cite{paper1} using the `robust decomposition lemma' of~\cite{Kelly}. 
As an `input' this lemma needs a suitable set of exceptional systems, which will be part of the decomposition found in this paper.

The nature of the decomposition of the exceptional edges into exceptional systems depends on the structure of the bipartite subgraph $G[A',B']$ of $G$:
we say that $G$ is `critical' if many edges of $G[A',B']$ are incident 
to very few (exceptional) vertices. In our decomposition into exceptional systems,
we will need to distinguish between (a) the non-critical case when $G[A',B']$ contains contains many edges,
(b) the critical case when $G[A',B']$ contains contains many edges, and (c) the case when $G[A',B']$ contains
only a few edges. The three lemmas guaranteeing this decomposition are the main results of this paper.
In these lemmas, we will be able to assume that $A_0$ and $B_0$ are independent sets of vertices, as suitable Hamilton cycles covering all edges 
of $G[A_0]$ and $G[B_0]$ are already found in~\cite{paper1}.

\section{Notation and tools}

\subsection{Notation}

Given a graph $G$, we write $V(G)$ for its vertex set, $E(G)$ for its edge set, $e(G):=|E(G)|$ for
the number of its edges and $|G|:=|V(G)|$ for the number of its vertices. 
We write $\delta(G)$ for the minimum degree of $G$, $\Delta(G)$ for its maximum degree
and $\chi'(G)$ for the edge-chromatic number of~$G$.
Given a vertex $v$ of $G$ and a set $A\subseteq V(G)$,
we write $d_G(v,A)$ for the number of all those neighbours of $v$ in $G$ which lie in~$A$. Given $A,B\subseteq V(G)$,
we write $e_G(A)$ for the number of all those edges of $G$ which have both endvertices in $A$ and $e_G(A,B)$ for the number of
\emph{$AB$-edges} of $G$, i.e.~for the number of all those edges of $G$ which have one endvertex in $A$ and its other endvertex in $B$. If $A\cap B=\emptyset$, we denote by $G[A,B]$ the bipartite subgraph of $G$
whose vertex classes are $A$ and $B$ and whose edges are all $AB$-edges of $G$. We often omit the index $G$ if the graph $G$ is clear from the context.

Given a vertex set $V$ and two edge-disjoint graphs $G$ and $H$ with $V(G),V(H)\subseteq V$, we write $G+H$ for the graph whose vertex
set is $V(G)\cup V(H)$ and whose edge set is $E(G)\cup E(H)$.%
	\COMMENT{Have replaced the definition of $G \cup H$ with $G+H$ since multigraphs do not appear in this paper. So this is different to papers 1 and 2.}
We write $G-H$ for the subgraph of $G$ which is obtained from $G$
by deleting all the edges in $E(G)\cap E(H)$.%
    \COMMENT{So we don't require that $H\subseteq G$ when using this notation.} 
Given $A\subseteq V(G)$, we write $G-A$ for the graph obtained from $G$ by deleting all vertices in~$A$.

We say that a graph $G$ has a \emph{decomposition} into $H_1,\dots,H_r$ if $G=H_1+\dots +H_r$ and the $H_i$ are pairwise
edge-disjoint.

A \emph{path system} is a graph $Q$ which is the union of vertex-disjoint paths (some of them might be trivial).
We say that $P$ is a \emph{path in Q} if $P$ is a component of $Q$ and, abusing the notation, sometimes write $P\in Q$ for this.
We often view a matching $M$ as a graph (in which every vertex has degree precisely one).%
    \COMMENT{This is different to eg the bipartite paper where a matching us a set of edges. CHECK whether we always use this
def in this paper, ie whether we write $e(M)$ for the number of edges and not $|M|$.}

In order to simplify the presentation, we omit floors and ceilings and treat large numbers as integers whenever this does
not affect the argument. The constants in the hierarchies used to state our results have to be chosen from right to left.
More precisely, if we claim that a result holds whenever $0<1/n\ll a\ll b\ll c\le 1$ (where $n$ is the order of the graph),
then this means that
there are non-decreasing functions $f:(0,1]\to (0,1]$, $g:(0,1]\to (0,1]$ and $h:(0,1]\to (0,1]$ such that the result holds
for all $0<a,b,c\le 1$ and all $n\in \mathbb{N}$ with $b\le f(c)$, $a\le g(b)$ and $1/n\le h(a)$. 
We will not calculate these functions explicitly. Hierarchies with more constants are defined in a similar way.
We will write $a = b \pm c$ as shorthand for $ b - c \le a \le b+c$.

\subsection{Tools} We will need the following Chernoff bound for binomial distribution
(see e.g.~\cite[Corollary 2.3]{Janson&Luczak&Rucinski00}).
Recall that the binomial random variable with parameters $(n,p)$ is the sum
of $n$ independent Bernoulli variables, each taking value $1$ with probability $p$
or $0$ with probability $1-p$.

\begin{prop}\label{prop:chernoff}
Suppose $X$ has binomial distribution and $0<a<3/2$. Then
$\mathbb{P}(|X - \mathbb{E}X| \ge a\mathbb{E}X) \le 2 e^{-a^2\mathbb{E}X /3}$.
\end{prop}

We will also use the following special cases of Propositions~6.1 and 6.3 in~\cite{paper1}%
	\COMMENT{AL: propositions instead of lemmas, added refencne }
 which, given a suitable graph $G$ and a
partition $A',B'$ of $V(G)$, provide bounds on the number
$e_G(A',B')$ of edges between $A'$ and~$B'$.

\begin{prop} \label{prp:e(A',B')}
Let $G$ be a graph on $n$ vertices with $\delta(G) \ge D$ and
let $A',B'$ be a partition of $V(G)$.
If $D \geq n - 2\lfloor n/4 \rfloor -1$, then $e_G (A',B') \geq D$ unless
$n = 0 \pmod4$, $D = n/2 -1$ and $|A'| = |B'| = n/2$.
\end{prop}

\begin{prop}  \label{prp:e(A',B')2}
Let $G$ be a $D$-regular graph on $n$ vertices with $D \ge \lfloor n/2 \rfloor$.
Let $A',B'$ be a partition of $V(G)$ with $|A'|,|B'| \ge D/2$ and $\Delta(G[A',B']) \le D/2$.
Then 
\begin{align}
e_{G - U}(A',B') \ge 
\begin{cases}
D - 28 & \textrm{if $D \ge n/2$,}\\
D/2 - 28 & \textrm{if $D = (n-1)/2$}
\end{cases}
\nonumber
\end{align}
for every $U \subseteq V(G)$ with $|U| \le 3$.
\end{prop}

Finally, we will also need the following result, which is a simple consequence of 
Vizing's theorem and was first  observed by McDiarmid and independently by de Werra (see e.g.~\cite{west}).%
\COMMENT{no need for original references as it's not that deep...}
\COMMENT{AL: deleted `A proof is also included in~\cite{paper2}.'}

\begin{prop} \label{prop:matchingdecomposition}
Let $G$ be a graph with $\chi '(G) \le m $. Then $G$ has a decomposition into $m$ matchings $M_1, \dots ,M_m$
with $|e(M_i) - e(M_j)| \le 1$ for all $i, j  \le m$.
\end{prop}

%%%%%%%%%%%%%%%%%%%%%%%%%%%%%%%%%%%%%%%%%%%%%%%%%%%%%%%%%%%%%%%%%%%%%%%%%%%%%%%%%%

\section{Exceptional systems, $(K,m,\eps_0)$-partitions and exceptional schemes}\label{sec:BES}

In this section, we formally introduce `exceptional (path) systems'.
Their first property is that the (interiors of) their paths cover all exceptional vertices.

Suppose that $A,A_0,B,B_0$ forms a partition of a vertex set $V$ of size $n$ such that $|A| = |B|$. Let $V_0:=A_0\cup B_0$.
An \emph{exceptional cover} $J$ is a graph which satisfies the following properties:
\begin{enumerate}[label={(EC{\arabic*})}]
\item $J$ is a path system with $V_0\subseteq V(J)\subseteq V$.%
   \COMMENT{One might set $V(J)=V$ instead. Or alternatively, $d_J(v)= 1$ for every $v \in V(J) \setminus V_0$.
However, this does not fit well with the bipartite case.}
\item $d_J(v) =2 $ for every $v \in V_0$ and $d_J(v) \le 1$ for every $v \in V(J) \setminus V_0$.
\item $e_J(A), e_J(B) = 0$.
\end{enumerate}
We say that $J$ is an \emph{exceptional system with parameter~$\eps_0$}, or an \emph{ES} for short, if $J$ satisfies the following properties:
\begin{enumerate}[label={(ES{\arabic*})}]
	\item $J$ is an exceptional cover.
	\item One of the following is satisfied:
	\begin{itemize}
	\item[(HES)] The number of $AB$-paths in $J$ is even and positive. In this case we say $J$ is a \emph{Hamilton exceptional system}, or \emph{HES} for short.
	\item[(MES)] $e_J(A',B')=0$. In this case we say $J$ is a \emph{matching exceptional system}, or \emph{MES} for short. 
\end{itemize}
	\item $J$ contains at most $\sqrt{\epszero} n $ $AB$-paths.
\end{enumerate}
Note that by~(EC2) every $AB$-path in $J$ must be a maximal path in~$J$. 
In~\cite{paper1} we will extend each Hamilton exceptional system $J$ into a Hamilton cycle using only
edges induced by $A$ and edges induced by~$B$. This is%
\COMMENT{AL: change the to is}
 the reason for condition~(HES) since the
number of $AB$-paths in $J$ corresponds to the number of genuine `connections' between $A$ and $B$.%
    \COMMENT{TO DO: The def of a MES had to change. So need to be careful that all knock on effects looked after.}
In~\cite{paper1}, matching exceptional systems will always be extended into two edge-disjoint perfect matchings.

In general, we construct an exceptional system by first choosing an exceptional system candidate (defined below) and then extending it to an
exceptional system. More precisely, suppose that $A,A_0,B,B_0$ forms a partition of a vertex set $V$.
Let $V_0:=A_0\cup B_0$. A graph $F$ is called an \emph{exceptional system candidate with parameter
$\eps_0$}, or an \emph{ESC} for short, if $F$ satisfies the following properties:
\begin{itemize}
	\item[(ESC1)] $F$ is a path system with $V_0\subseteq V(F)\subseteq V$ and such that $e_F(A), e_F(B) = 0$.
	\item[(ESC2)] $d_F(v) \le 2$ for all $v \in V_0$ and $d_F(v) = 1$ for all $v \in V(F) \setminus V_0$.
	\item[(ESC3)] $e_F(A',B') \le \sqrt{ \epszero}  n/2$. In particular, $|V(F) \cap A|, |V(F) \cap B| \le 2|V_0| + \sqrt{ \epszero } n/2$.
	\item[(ESC4)] One of the following holds:
\begin{itemize}
\item[(HESC)] Let $b(F)$ be the number of maximal paths in $F$ with one endpoint in $A'$ and the other in $B'$. Then $b(F)$ is even
and $b(F)>0$. In this case we say that $F$ is a \emph{Hamilton exceptional system candidate}, or \emph{HESC} for short.
\item[(MESC)] $e_F(A',B') =0$. In this case, $F$ is called a \emph{matching exceptional system candidate} or \emph{MESC} for short.
\end{itemize}
\end{itemize}
Note that if $d_F(v) = 2$ for all $v \in V_0$, then $F$ is an exceptional system. 
Also, if $F$ is a Hamilton exceptional system candidate with $e(F) =2$, then $F$ consists of two independent $A'B'$-edges.
Moreover, note that (EC2) allows an exceptional cover~$J$
(and so also an exceptional system~$J$) to contain vertices in $A\cup B$ which are isolated in~$J$. However, (ESC2) does not allow for this
in an exceptional system candidate~$F$.%
    \COMMENT{The latter is quite handy since it allows us to write
$|V(F) \cap A|$ for the number of vertices in $A$ which a incident to an edge of $F$. The former ensures that things are the same as in the
bipartite case for (balanced) exceptional systems and so in the approx paper we don't use two different versions...}

Similarly to condition (HES), in (HESC) the parameter $b(F)$ counts the number of `connections' between $A'$ and $B'$.
In order to extend a Hamilton exceptional system candidate into a Hamilton cycle without using any additional $A'B'$-edges, it is clearly necessary that $b(F)$
is positive and even.

The following result shows that we can extend an exceptional system candidate into a exceptional system by adding
suitable $A_0A$- and $B_0B$-edges. Its easy proof is included in~\cite[Lemma~7.2]{paper1}.%
\COMMENT{AL: added Lemma 7.2}
\begin{lemma} \label{lma:ESextend}
Suppose that $ 0 < 1/n \ll \epszero \ll 1$ and that $n \in \mathbb{N}$.
Let $G$ be a graph on $n$ vertices so that
\begin{itemize}
\item[{\rm (i)}] $A,A_0,B,B_0$ forms a partition of $V(G)$ with $|A_0 \cup B_0| \le \eps_0 n$;
\item[{\rm (ii)}] $d(v,A) \ge \sqrt{\epszero} n$ for all $v \in A_0$ and $d(v,B) \ge \sqrt{\epszero} n$ for all $v \in B_0$.
\end{itemize}
Let $F$ be an exceptional system candidate with parameter $\eps_0$. 
Then there exists an exceptional system $J$ with parameter $\eps_0$ such that $F\subseteq J\subseteq G +  F$%
	\COMMENT{TO DO: check we have changed $F\subseteq J\subseteq G \cup  F$ to $F\subseteq J\subseteq G +  F$ in paper I.}
and such that every edge of $J-F$ lies in $G[A_0,A]+G[B_0,B]$.
Moreover, if $F$ is a Hamilton exceptional system candidate, then $J$ is a Hamilton exceptional system.
Otherwise $J$ is a matching exceptional system.
\end{lemma}

As mentioned earlier, the exceptional systems we seek will need to be `localized'. For a formal definition,  
let $K,m\in\mathbb{N}$ and $\eps_0>0$.
A \emph{$(K,m,\eps_0)$-partition $\mathcal{P}$} of a set $V$ of vertices is a partition of $V$ into sets $A_0,A_1,\dots,A_K$
and $B_0,B_1,\dots,B_K$ such that $|A_i|=|B_i|=m$ for all $i\ge 1$ and $|A_0\cup B_0|\le \eps_0 |V|$.
The sets $A_1,\dots,A_K$ and $B_1,\dots,B_K$ are called \emph{clusters} of $\mathcal{P}$
and $A_0$, $B_0$ are called \emph{exceptional sets}. We often write $V_0$ for $A_0\cup B_0$ and think of the
vertices in $V_0$ as `exceptional vertices'. Unless stated otherwise, whenever $\mathcal{P}$ is a $(K,m,\eps_0)$-partition,
we will denote the clusters by $A_1,\dots,A_K$ and $B_1,\dots,B_K$ and the exceptional sets by $A_0$ and $B_0$.
We will also write $A:=A_1\cup\dots\cup A_K$, $B:=B_1\cup\dots\cup B_K$, $A':=A_0\cup A_1\cup\dots\cup A_K$
and $B':=B_0\cup B_1\cup\dots\cup B_K$.

Given a $(K,m, \epszero)$-partition $\mathcal{P}$ and $1\le i,i' \le K$, we say that $J$ is an \emph{$ (i,i')$-localized Hamilton exceptional system}
(abbreviated as \emph{$(i,i')$-HES}) if $J$ is a Hamilton exceptional system and $V(J)\subseteq V_0 \cup A_{i} \cup B_{i'}$.
In a similar way, we define
\begin{itemize}
\item \emph{$ (i,i')$-localized matching exceptional systems} (\emph{$(i,i')$-MES}),
\item \emph{$ (i,i')$-localized exceptional systems} (\emph{$(i,i')$-ES}),
\item \emph{$ (i,i')$-localized Hamilton exceptional system candidates} (\emph{$(i,i')$-HESC}),
\item \emph{$ (i,i')$-localized matching exceptional system candidates} (\emph{$(i,i')$-MESC}),
\item \emph{$ (i,i')$-localized exceptional system candidates} (\emph{$(i,i')$-ESC}). 
\end{itemize}
To make clear with which partition we are working, we sometimes also say that
$J$ is an $ (i,i')$-localized Hamilton exceptional system with respect to $\cP$ etc.

Finally, we define an `exceptional scheme', which will be the structure within which
we find our localized exceptional systems. Given a graph $G$ on $n$ vertices and a partition $\mathcal P$ of $V(G)$, we call
$(G, \mathcal{P})$ a \emph{$(K,  m, \epszero,\eps)$-exceptional scheme} if the following properties are satisfied:
\begin{enumerate}[label={(ESch{\arabic*})}]
	\item $\mathcal P$ is a $(K, m, \epszero)$-partition of $V(G)$.
	\item $e(A),e(B) = 0$.
	\item If $v\in A$ then $d(v, B')< \epszero n$ and if $v\in B$ then $d(v, A')< \epszero n$.%
	\COMMENT{Previously had "If $d(v,A'), d(v,B') \ge \epszero n$, then $v \in V_0$". But this doesn't seem to be strong enough.
Also, note that there is no condition on $\Delta(G[A',B'])$ in the definition of exceptional scheme.}
	\item For all $v \in V(G)$ and all $1\le i \le K$ we have
    $d(v, A_i) = ( d(v,A) \pm\eps n ) / K$ and $d(v,B_i) = ( d(v,B) \pm\eps n ) / K$.
	\item For all $1\le i,i' \le K$ we have
	\begin{align*}
	&e(A_0,A_i)  = ( e(A_0,A) \pm \eps \max \{ e(A_0,A) , n \}     )/K ,\\
   &e(B_0,A_i)  = ( e(B_0,A) \pm \eps \max \{ e(B_0,A) , n \}     )/K ,\\
	&e(A_0,B_i)  = ( e(A_0,B) \pm \eps \max \{ e(A_0,B) , n \}     )/K ,\\
	&e(B_0,B_i)  = ( e(B_0,B) \pm \eps \max \{ e(B_0,B) , n \}     )/K ,\\
	&e(A_i,B_{i'})  = ( e(A,B) \pm \eps \max \{ e(A,B) , n \}     )/K^2 .
	\end{align*}
\end{enumerate}

%%%%%%%%%%%%%%%%%%%%%%%%%%%%%%%%%%%%%%%%%%%%%%%%%%%%%%%%%%%%%%%%%%%%%%%%%%%%%%%%
\section{Constructing localized exceptional systems}\label{sec:locES}

Given a $D$-regular graph $G$ and a $(K, m, \epszero)$-partition $\mathcal{P}$ of $V(G)$, let $G':=G-G[A]-G[B]$
and suppose that $(G',\mathcal{P})$ is an exceptional scheme.
Roughly speaking, the aim of this section is to decompose $G'$ into edge-disjoint exceptional systems.
In~\cite{paper1}, each of these exceptional systems $J$ will then be extended into a Hamilton cycle
(in the case when $J$ is a Hamilton exceptional system) or into
two perfect matchings (in the case when $J$ is a matching exceptional system). We will ensure that all but a small number of these
exceptional systems are localized (with respect to $\mathcal{P}$). 

Rather than decomposing $G'$ in a single step, we actually need to proceed in two steps:
initially, we find a small number of exceptional systems $J$ which have some additional useful properties (e.g.~the number of $A'B'$-edges of $J$
is either zero or two). In~\cite{paper1} these exceptional systems will be used to construct the robustly decomposable graph $G^{\rm rob}$.
(Recall that the role of $G^{\rm rob}$ in~\cite{paper1} was also discussed in Section~\ref{sketch}.)
Some of the additional properties of the exceptional systems contained in $G^{\rm rob}$ then allow us to 
find the desired decomposition of  $G^{\diamond} := G'- G^{\rm rob}$ into exceptional systems.

In order to construct the required (localized) exceptional systems, we will distinguish three cases:
\begin{itemize}
\item[(a)] the case when $G$ is `non-critical' and contains at least $D$ $A'B'$-edges (see Lemma~\ref{lma:BESdecom}
in Section~\ref{noncritical}); 
\item[(b)] the case when $G$ is `critical' and contains
at least $D$ $A'B'$-edges (see Lemma \ref{lma:BESdecomcritical} in Section~\ref{sec:critical});
\item[(c)] the case when $G$ contains less than $D$ $A'B'$-edges (see Lemma~\ref{lma:PBESdecom} in Section~\ref{1factsec}).  
\end{itemize}
Each of the three lemmas above is formulated in such a way that we can apply it twice in~\cite{paper1}:
firstly to obtain the small number of exceptional systems needed for the robustly decomposable graph $G^{\rm rob}$ and secondly for the 
decomposition of the graph $G^{\diamond}:=G-G^{\rm rob}-G[A]-G[B]$ into exceptional systems.

\subsection{Critical graphs}

Let $G$ be a $D$-regular graph and let $A',B'$ be a partition of $V(G)$. 
Roughly speaking, $G$ is critical if most of its $A'B'$-edges are incident to only a few vertices. More precisely, we say that
$G$ is \emph{critical} (with respect to $A',B'$ and $D$) if both of the following hold:
\begin{itemize}
\item $\Delta(G[A',B']) \ge 11 D/40$;
\item  $e(H) \le 41 D/40$ for all subgraphs $H$ of $G[A',B']$ with $\Delta(H) \le 11 D/40$.%
     \COMMENT{Note that there is no assumption on $e_G(A',B')$.}
\end{itemize}
%Note that the property of $G$ being critical depends only on $D$ and the partition $A',B'$ of $V(G)$.
One example of a critical graph is the following:
$G_{\rm crit}$ consists of two disjoint cliques on $(n-1)/2$ vertices with vertex set $A$ and $B$ respectively, where $n=4k+1$ for some $k \in \mathbb{N}$.
In addition, there is a vertex $a$ which is adjacent to exactly half of the vertices in each of $A$ and $B$.
Also, add a perfect matching $M$ between those vertices of $A$ and those vertices in $B$ not adjacent to $a$.
Let  $A':=A \cup \{a\}$, $B':=B$ and $D:=(n-1)/2$. 
Then $G_{\rm crit}$ is critical, and $D$-regular with $e(A',B')=D$. Note that $e(M)=D/2$.

To obtain a Hamilton decomposition of $G_{\rm crit}$, 
we will need to decompose $G_{\rm crit}[A',B']$ into $D/2$ Hamilton exceptional system candidates $F_s$
(which need to be matchings of size exactly two in this case).
In this example, this decomposition is essentially unique:
every $F_s$ has to consist of exactly one edge in $M$ and one edge incident to $a$.
Note that in this way, every edge between $a$ and $B$ yields a `connection'  (i.e.~a maximal path) between $A'$ and $B'$ required in (ESC4).

The following lemma collects some properties of critical graphs. In particular, there is a set $W$ consisting of  between one and three vertices
with many neighbours in both $A$ and $B$ (such as the vertex $a$ in~$G_{\rm crit}$ above).
As in the example of $G_{\rm crit}$, we will need to use $A'B'$-edges incident to one or two vertices in $W$
to provide connections between $A'$ and $B'$ when constructing 
the Hamilton exceptional system candidates in the critical case~(b).

\begin{lemma} \label{critical}
Suppose that $0< 1/n \ll 1$ and that $D, n \in \mathbb N$ are such that
\begin{equation} \label{minexact}
D\ge n - 2\lfloor n/4 \rfloor -1=
\begin{cases}
n/2-1 & \textrm{if $n = 0 \pmod 4$,}\\
(n-1)/2 & \textrm{if $n = 1 \pmod 4$,}\\
n/2 & \textrm{if $n = 2 \pmod 4$,}\\
(n+1)/2 & \textrm{if $n = 3 \pmod 4$.}
\end{cases}
\end{equation}
Let $G$ be a $D$-regular graph on $n$ vertices and let $A',B'$ be a partition of $V(G)$ with $|A'|,|B'| \ge D/2$ and $\Delta(G[A',B']) \le D/2$.
Suppose that $G$ is critical.
Let $W$ be the set of vertices $w \in V(G)$ such that $d_{G[A',B']}(w) \ge 11D/40$.
Then the following properties are satisfied:%
%     \COMMENT{TO DO: Previously the lemma also included (vi) "If $|W| = 3$, then $e_{G-\{ w_1,w_2\}}(A',B') \le 7D/10+5$ for any two distinct vertices $w_1,w_2 \in W$."
% But it seems that we never use this. CHECK!}
\begin{itemize}
	\item[$ \rm (i)$] $1 \le |W| \le 3$.
	\item[$ \rm (ii)$] Either $D = (n-1)/2$ and $n = 1 \pmod{4}$, or $D = n/2-1$ and $n = 0 \pmod{4}$.
	Furthermore, if $n =1 \pmod{4}$, then $|W|=1$.
	\item[$\rm (iii)$] $e_{G}(A',B') \le 17D/10+5 < n$.
	\item[$\rm (iv)$] \begin{align*}
e_{G- W}(A',B') \le 
\begin{cases}
3D/4+5	& \textrm{if $|W|=1$,} \\
19D/{40}+5	& \textrm{if $|W|=2$,} \\
D/5+5	& \textrm{if $|W|=3$.}
\end{cases}
\end{align*}
	\item[$ \rm (v)$] There exists a set $W'$ of vertices such that $W \subseteq W'$, $|W'| \le 3$ and for all $w' \in W'$ and $v \in V(G) \setminus W'$
	we have 
\begin{align*}
d_{G[A',B']}(w') & \ge \frac{21D}{80}, & 
d_{G[A',B']}(v) & \le \frac{11D}{40} &
{\rm and} & &
d_{G[A',B']}(w') - d_{G[A',B']}(v)& \ge \frac{D}{240}.
\end{align*}
	\end{itemize}
\end{lemma}
\begin{proof}
Let $w_1, \dots, w_4$ be vertices of $G$ such that 
\begin{align*}
d_{G[A',B']}(w_1) \ge \dots \ge d_{G[A',B']}(w_4) \ge d_{G[A',B']}(v)
\end{align*}
for all $v \in V(G) \setminus \{w_1, \dots, w_4\}$. Let $W_4:=\{w_1, \dots, w_4\}$. Suppose that $d_{G[A',B']}(w_4) \ge 21D/80$.
Let $H$ be a spanning subgraph of $G[A',B']$ such that
$d_H(w_i)= \lceil 21D/80\rceil$ for all $i\le 4$ and such that every vertex $v\in V(G)\setminus W_4$ satisfies
$N_H(v)\subseteq W_4$.%
    \COMMENT{Such a $H$ can be obtained by starting with $G[A'\cap W_4,B'\cap W_4]$ and adding suitable further edges.}
Thus $\Delta(H)= \lceil 21D/80\rceil$ and so $e(H) \le 41D/40$ since $G$ is critical. On the other hand, $e(H)\ge 4\cdot \lceil 21D/80\rceil -4$,
a contradiction. (Here we subtract four to account for the edges of $H'$ between vertices in $W$.)
Hence, $d_{G[A',B']}(w_4) < 21 D /80$ and so $|W|\le 3$. But $|W|\ge 1$ since $G$ is critical. So (i) holds.

Let $j$ be minimal such that
$d_{G[A',B']}(w_j)\le 21D/80$. So $1<j\le 4$. Choose an index $i$ with $1\le i<j$ such that
$W\subseteq \{w_1,\dots,w_i\}$ and $d_{G[A',B']}(w_i) - d_{G[A',B']}(w_{i+1}) \ge D/240$.
Then the set $W':=\{w_1,\dots,w_i\}$ satisfies~(v).

Let $H'$ be a spanning subgraph of $G[A',B']$ such that $G[A'\setminus W,B'\setminus W]\subseteq H'$ and
$d_{H'}(w)= \lfloor 11D/40 \rfloor$ for all $w\in W$.%
    \COMMENT{Such a $H'$ can be obtained by starting with $G[A'\cap W,B'\cap W]+ G[A'\setminus W,B'\setminus W]$ and adding
suitable further edges between $W$ and $V(G)\setminus W$.}
Similarly as before, $e(H') \le 41D/40$ since $G$ is critical. Thus 
\begin{align*}
	41D/40 & \ge e(H') \ge e ( H'- W ) + \lfloor 11D / 40 \rfloor |W| - 2\\
	& = e_{G- W}(A',B') + \lfloor 11D / 40 \rfloor |W| - 2.
\end{align*}
This in turn implies that
\begin{align}
e_{G- W}(A',B') & \le (41 - 11 |W|)D / 40 + 5 \label{eq:e(HnotW)}.
\end{align}
Together with~(i) this implies~(iv).
If $D \ge n/2$, then by Proposition~\ref{prp:e(A',B')2} we have $e_{G- W}(A',B')  \ge D - 28$.
This contradicts~(iv). Thus~(\ref{minexact}) implies that $D = (n-1)/2$ and $n = 1 \pmod{4}$, or $D = n/2-1$ and $n = 0 \pmod{4}$.
If $n = 1 \pmod 4$ and $D=(n-1)/2$, then Proposition~\ref{prp:e(A',B')2} implies that $e_{G-W}(A',B')\ge D/2-28$.
Hence, by (iv) we deduce that $|W| =1$ and so (ii) holds.
Since $|W| \le 3$ and $\Delta(G[A',B']) \le D/2$, we have
\begin{align*}
	e_G (A',B') & \le e_{G- W}(A',B') + \frac{|W| D}2
		 \stackrel{(\ref{eq:e(HnotW)})}{\le} \frac{(41+9|W|)D}{40}+5 \le \frac{17D}{10}+5 <n.
\end{align*}
(The last inequality follows from~(ii).) This implies~(iii).
\end{proof}

\subsection{Non-critical case with $e(A',B') \ge D$.} \label{noncritical}
Recall from the beginning of Section~\ref{sec:locES} that our aim is
to find a decomposition of $G - G[A] -G[B]$ into suitable exceptional systems (in particular, most of
these exceptional systems have to be localized). The following lemma implies that this can be done
if $G$ is not critical and $e(A',B') \ge D$. We will prove this lemma in this subsection and
will then consider the remaining two cases in the next two subsections.%

\begin{lemma} \label{lma:BESdecom}
Suppose that $0 <  1/n  \ll \epszero \ll  \eps \ll   \lambda, 1/K \ll 1$, that $D\ge n/3$, that $0 \le \phi  \ll 1$
and that%
    \COMMENT{Previously also had that $\phi n \in \mathbb{N}$. But this follows since $D,K, (D - \phi n)/(2K^2) \in \mathbb{N}$.}
$D, n, K, m, \lambda n/K^2, (D - \phi n)/(2K^2) \in \mathbb{N}$. Suppose that the following conditions hold:
\begin{itemize}
	\item[{\rm(i)}] $G$ is a $D$-regular graph on $n$ vertices.%
\COMMENT{Note that $D$ satisfies $1/3 \le D/n < 1$ not $D \ge n - 2\lfloor n/4\rfloor -1$.}
	\item[{\rm(ii)}] $\mathcal{P}$ is a $(K, m, \epszero)$-partition of $V(G)$ such that $D \le e_G(A',B') \le \epszero n^2$ and $\Delta(G[A',B']) \le D/2$.
Furthermore, $G$ is not critical.
	\item[{\rm(iii)}] $G_0$ is a subgraph of $G$ such that $G[A_0]+G[B_0] \subseteq G_0$, $e_{G_0}(A',B') \leq \phi n$ and $d_{G_0}(v) = \phi n $ for all $v \in V_0$.
	\item[{\rm(iv)}] Let $G^{\diamond} := G - G[A] - G[B] -G_0$.  $e_{G^\diamond} (A',B')$ is even and $(G^{\diamond}, \mathcal{P})$
is a $(K, m, \epszero,\eps)$-exceptional scheme.
\end{itemize}
Then there exists a set $\mathcal{J}$ consisting of $(D-\phi n)/2$ edge-disjoint Hamilton exceptional systems with parameter $\eps_0$ in $G^{\diamond}$ which satisfies the
following properties: 
\begin{itemize}
    \item[\rm (a)] Together all the Hamilton exceptional systems in $\mathcal{J}$ cover all edges of $G^{\diamond}$.
	\item[\rm (b)] For all $1\le i,i' \le K$, the set $\mathcal{J}$ contains $(D - (\phi + 2 \lambda) n)/(2K^2)$ $(i,i')$-HES.
	Moreover, $\lambda n /K^2$ of these $(i,i')$-HES $J$ are such that $e_J(A',B') =2$.
\end{itemize}
\end{lemma}

Note that (b) implies that $\mathcal{J}$ contains $\lambda n$ Hamilton exceptional systems which might not be localized.
On the other hand, the lemma is `robust' in the sense that we can remove a sparse subgraph $G_0$ before we find the decomposition $\mathcal{J}$
into Hamilton exceptional systems. (In particular, as discussed at the beginning of the section, we can remove the graph $G^{\rm rob}$ before applying the lemma.)

We will split the proof of Lemma~\ref{lma:BESdecom} into the following four steps:
\begin{itemize}
	\item[\bf Step 1] We first decompose $G^{\diamond}$ into edge-disjoint `localized' subgraphs $H(i,i')$ and $H'(i,i')$ (where $1\le i,i' \le K$).
	More precisely, each $H(i,i')$ only contains $A_0A_i$-edges and $B_0B_{i'}$-edges of $G^\diamond$ while all edges of $H'(i,i')$
lie in $G^{\diamond}[A_0 \cup A_i,  B_0 \cup B_{i'}]$, and all the edges of $G^\diamond$ are distributed evenly amongst the $H(i,i')$ and $H'(i,i')$
(see Lemma~\ref{lma:randomslice}). 
We will then move a small number of $A'B'$-edges between the $H'(i,i')$ in order to obtain graphs $H''(i,i')$ such that $e(H''(i,i'))$ is even (see Lemma~\ref{lma:move}).
	\item[\bf Step 2] We decompose each $H''(i,i')$ into $(D - \phi n )/(2K^2)$ Hamilton exceptional system candidates (see Lemma~\ref{lma:BESdecomprelim}).
	\item[\bf Step 3] Most of the Hamilton exceptional system candidates constructed in Step~2 will be extended into an $(i,i')$-HES (see Lemma~\ref{lma:BESextend2}).
	\item[\bf Step 4] The remaining Hamilton exceptional system candidates will be extended into Hamilton exceptional systems, which need not be localized
	(see Lemma~\ref{globalBES}).
	(Altogether, these will be the $\lambda n$ Hamilton exceptional systems in $\mathcal{J}$ which are not mentioned in Lemma~\ref{lma:BESdecom}(b).)
\end{itemize}

\subsubsection{Step $1$: Constructing the graphs $H''(i,i')$}

The next lemma from~\cite[Lemma~9.2]{paper1}%
\COMMENT{AL: added Lemma 9.2}
 will be used to find a
decomposition of $G^\diamond$ into suitable `localized subgraphs' $H(i,i')$ and $H'(i,i')$ as decribed in Step~1 above.

\begin{lemma} \label{lma:randomslice}
Suppose that $0 <  1/n  \ll \epszero \ll  \eps  \ll 1/K \ll 1$ and that $n, K,m\in \mathbb{N}$.
Let $(G, \mathcal{P})$ be a $(K, m, \epszero, \eps )$-exceptional scheme with $|G|=n$ and $e_G(A_0),e_G(B_0)  =0$.
Then $G$ can be decomposed into edge-disjoint spanning subgraphs $H(i,i')$ and $H'(i,i')$ of $G$ (for all $1 \le i,i' \le K$)%
	\COMMENT{AL: changed $i,i' \le K$ to  $1 \le i,i' \le K$.\\
TO DO: same for paper 1, Lemma 9.2}
such that the following properties hold, where $G(i,i'):=H(i,i')+H'(i,i')$:
\begin{itemize}
\item[\rm (a$_1$)] Each $H(i,i')$ contains only $A_0A_i$-edges and $B_0B_{i'}$-edges.
\item[\rm (a$_2$)] All edges of $H'(i,i')$ lie in $G[A_0 \cup A_i, B_0 \cup B_{i'}]$.
\item[\rm (a$_3$)] $e ( H'(i,i') )   =  ( e_{G}(A',B')  \pm  4 \eps \max \{ n, e_{G}(A',B') \}) /K^2$.
\item[\rm (a$_4$)] $d_{H'(i,i')}(v )  =  ( d_{G[A',B']}(v)  \pm 2\eps n)/K^2$ for all $v \in V_0$.
\item[\rm (a$_5$)] $d_{G(i,i')}(v )  =  ( d_{G}(v)  \pm 4 \eps n)/K^2$ for all $v \in V_0$.
\end{itemize}
\end{lemma}
Let $H(i,i')$ and $H'(i,i')$ be the graphs obtained by applying Lemma~\ref{lma:randomslice} to $G^{\diamond}$.
As mentioned before, we would like to decompose each $H'(i,i')$ into Hamilton exceptional system candidates.
In order to do this, $e(H'(i,i'))$ must be even. The next lemma shows that we can ensure this property
without destroying the other properties of the $H'(i,i')$ too much by moving a small number of edges 
between the $H'(i,i')$.

\begin{lemma} \label{lma:move}
Suppose that $0 <  1/n  \ll \epszero \ll  \eps \ll   \eps' \ll \lambda, 1/K \ll 1$, that $D\ge n/3$, that
$0 \le \phi  \ll 1$ and that $D, n, K, m,  (D - \phi n)/(2K^2) \in \mathbb{N}$.
Define $\alpha$ by 
\begin{align} \label{alpha}
2 \alpha n := \frac{ D - \phi n }{K^2}  \ \  \ \ \ \ \text{and let} \ \ \ \  \  \ 
\gamma  : = \alpha -  \frac{2\lambda}{K^2}.
\end{align}
Suppose that the following conditions hold:
\begin{itemize}
	\item[$ \rm (i)$] $G$ is a $D$-regular graph on $n$ vertices.
	\item[$ \rm (ii)$] $\mathcal{P}$ is a $(K, m, \epszero)$-partition of $V(G)$ such that $D \le e_G(A',B') \le \epszero n^2$ and $\Delta(G[A',B']) \le D/2$.
Furthermore, $G$ is not critical.
	\item[$ \rm (iii)$] $G_0$ is a subgraph of $G$ such that $G[A_0]+G[B_0] \subseteq G_0$, $e_{G_0}(A',B') \leq \phi n$ and $d_{G_0}(v) = \phi n $ for all $v \in V_0$.
	\item[$ \rm (iv)$] Let $G^{\diamond} := G - G[A] - G[B] -G_0$.  $e_{G^\diamond} (A',B')$ is even and $(G^{\diamond}, \mathcal{P})$ is a $(K, m, \epszero,\eps)$-exceptional scheme.
\end{itemize}
Then $G^{\diamond}$ can be decomposed into edge-disjoint spanning subgraphs $H(i,i')$ and $H''(i,i')$ of $G^{\diamond}$ (for all $1 \le i,i' \le K$)%
	\COMMENT{AL: changed $i,i' \le K$ to  $1 \le i,i' \le K$.}
such that the following properties hold, where $G'(i,i'):=H(i,i')+H''(i,i')$:
\begin{itemize}
\item[\rm (b$_1$)] Each $H(i,i')$ contains only $A_0A_i$-edges and $B_0B_{i'}$-edges.
\item[\rm (b$_2$)] $H''(i,i')\subseteq G^{\diamond}[A',B']$. Moreover,
all but at most $\eps' n$ edges of $H''(i,i')$ lie in $G^{\diamond}[A_0 \cup A_i, B_0 \cup B_{i'}]$.
\item[\rm (b$_3$)] $e(H''(i,i'))$ is even and $ 2 \alpha n \le e(H''(i,i')) \le 11\epszero n^2/(10K^2)$.
\item[\rm (b$_4$)] $\Delta(H''(i,i')) \le 31 \alpha n/30 $.
\item[\rm (b$_5$)] $d_{G'(i,i')}(v )  =  \left( 2 \alpha \pm   \eps' \right) n $ for all $v \in V_0$.
\item[\rm (b$_6$)] Let $\widetilde{H}$ be any spanning subgraph of $H''(i,i')$ which maximises $e(\widetilde{H})$
under the constraints that $\Delta(\widetilde{H}) \le 3\gamma n /5$, $H''(i,i')[A_0,B_0] \subseteq \widetilde{H}$ and $e(\widetilde{H})$ is even.
Then $e(\widetilde{H}) \ge 2 \alpha n $.
\end{itemize}
\end{lemma}
\begin{proof}
Since $\phi  \ll 1/3\le D/n$, we deduce that%
    \COMMENT{$\gamma=\alpha -2\lambda/K^2\ge \alpha -2\lambda \cdot 7\alpha=(1-14\lambda)\alpha$}
\begin{align} \label{alphahier}
	\alpha \ge 1/(7K^2), \ \ \ \ \ (1-14\lambda)\alpha\le \gamma<\alpha \ \ \ \ \text{and} \ \ \ \ \eps \ll \eps'  \ll \lambda, 1/K,\alpha, \gamma \ll 1.
\end{align}
Note that (ii) and (iii) together imply that
\begin{align}	\label{alpha1}
e_{G^{\diamond}}(A',B') \ge D - \phi n \stackrel{(\ref{alpha})}{=}  2 K^2 \alpha n \stackrel{(\ref{alphahier})}{\ge } n/4.
\end{align}
By (i) and (iii), each $v \in V_0$ satisfies
\begin{equation}\label{eq:degGdiam}
d_{G^{\diamond}}(v) = D - \phi n \stackrel{(\ref{alpha})}{=} 2 K^2\alpha n.
\end{equation}
Apply Lemma~\ref{lma:randomslice} to decompose $G^{\diamond}$ into subgraphs $H(i,i')$, $H'(i,i')$ (for all $1\le i,i' \le K$)
satisfying the following properties, where $G(i,i'):=H(i,i')+H'(i,i')$:
\begin{itemize}
\item[(a$_1'$)] Each $H(i,i')$ contains only $A_0A_i$-edges and $B_0B_{i'}$-edges.
\item[(a$_2'$)] All edges of $H'(i,i')$ lie in $G^{\diamond}[A_0 \cup A_i, B_0 \cup B_{i'}]$.
\item[(a$_3'$)] $e ( H'(i,i') )   =   (1 \pm 16\eps)  e_{G^\diamond}(A',B') /K^2$.
	In particular, 
\begin{align*}
	 2 (1- 16\eps) \alpha n \le e(H'(i,i')) \le (1+16 \eps) \epszero n^2/K^2.
\end{align*}
\item[(a$_4'$)] $d_{H'(i,i')}(v )  =  ( d_{G^\diamond[A',B']}(v)  \pm 2 \eps n)/K^2$ for all $v \in V_0$.
\item[(a$_5'$)] $d_{G(i,i')}(v )  =  ( 2 \alpha  \pm 4 \eps/K^2)n$ for all $v \in V_0$.
\end{itemize}
Indeed, (a$'_3$) follows from (\ref{alpha1}), Lemma~\ref{lma:randomslice}(a$_3$) and~(ii),
while (a$'_5$) follows from (\ref{eq:degGdiam}) and Lemma~\ref{lma:randomslice}(a$_5$).
We now move some $A'B'$-edges of $G^\diamond$ between the $H'(i,i')$ such that the graphs $H''(i,i')$
obtained in this way satisfy the following conditions:
\begin{itemize}
\item Each $H''(i,i')$ is obtained from $H'(i,i')$ by adding or removing at most $32K^2 \eps \alpha n\le \sqrt{\eps}n$ edges.
\item $e(H''(i,i'))\ge 2 \alpha n$ and $e(H''(i,i'))$ is even.
\end{itemize}
Note that this is possible by~(a$_3'$) and since $\alpha n\in\mathbb{N}$ and $e_{G^\diamond} (A',B') \geq 2K^2\alpha n$ is even by~(iv).

We will show that the graphs $H(i,i')$ and $H''(i,i')$ satisfy conditions (b$_1$)--(b$_6$).
Clearly both (b$_1$) and (b$_2$) hold.
(a$_3'$) implies that
\begin{equation} \label{eH'}
e(H''(i,i')) = (1\pm 16 \eps) e_{G^{\diamond}}(A',B')/K^2\pm \sqrt{\eps}n \overset{(\ref{alphahier}),(\ref{alpha1}) }{=}
(1\pm \eps') e_{G^{\diamond}}(A',B')/K^2.
\end{equation}
Together with (ii) and our choice of the $H''(i,i')$ this implies (b$_3$). (b$_5$) follows from (a$'_5$)
and the fact that $d_{G'(i,i')}(v )  =d_{G(i,i')}(v ) \pm \sqrt{\eps}n$.
Similarly, (a$_4'$) implies that for all $v \in V_0$ we have
\begin{equation}\label{dH'}
d_{H''(i,i')}(v )  =  ( d_{G^{\diamond} [A',B']}(v)  \pm \eps' n)/{K^2}.
\end{equation}
Recall that $\Delta(G[A',B']) \le D/2$ by~(ii). Thus
%%\begin{eqnarray*}
$$
	\Delta(H''(i,i'))  \overset{\eqref{dH'}}{\le}   %\frac{ \Delta(G[A',B']) + \eps' n}{K^2} 
%\le 
\frac{D/2 + \eps' n   }{K^2}   \overset{\eqref{alpha}}{=}  \left( \alpha +  \frac{\phi + 2\eps' }{2K^2} \right) n 
 \overset{\eqref{alphahier}}{\le}  \frac{31 \alpha n }{30},
$$
%\end{eqnarray*}
so (b$_4$) holds.

So it remains to verify (b$_6$). To do this, fix $1 \le i, i' \le K$%
\COMMENT{AL: changed $i,i' \le K$ to  $1 \le i,i' \le K$.}
 and set $H'' :=H''(i,i')$.
Let $\widetilde{H}$ be a subgraph of $H''$ as defined in (b$_6$). We need to show that $e(\widetilde{H}) \ge 2 \alpha n$.
Suppose the contrary that $e(\widetilde{H}) < 2 \alpha n$.
We will show that this contradicts the assumption that $G$ is not critical.
Roughly speaking, the argument will be that if $\widetilde{H}$ is sparse, then so is~$H''$.
This in turn implies that $G^\diamond$ is also sparse, and thus any subgraph of $G[A',B']$ of comparatively small maximum degree is also sparse, 
which leads to a contradiction.

Let $X$ be the set of all those vertices $x$ for which $d_{\widetilde{H}}(x) \geq  3 \gamma n/ 5 -2$. So $X\subseteq V_0$ by~(iv) and (ESch3).
Note that if $X = \emptyset$, then  $\widetilde{H} = H''$ and%
   \COMMENT{If $\widetilde{H} \neq H''$ then $e(H'')-e(\widetilde{H}) \ge 2$ (as it is even). But the maximality of $\widetilde{H}$ implies that
 adding any two edges in $H''-\widetilde{H}$ to $\widetilde{H}$ would create a vertex
of degree $>3 \gamma n/ 5$.}
so $e(\widetilde{H}) \ge  2 \alpha n$ by (b$_3$). If $|X| \ge 4$, then $e(\widetilde{H}) \ge 4 (  3 \gamma n/ 5  -2)  -4\ge 2\alpha n $
by~\eqref{alphahier}. Hence $1\le |X| \le 3$. Note that $\widetilde{H}-X$ contains all but at most one edge from $H''-X$.%
   \COMMENT{Might have deleted one edge in $H''-X$ in order to ensure that $e(\widetilde{H})$ is even.}
Together with the fact that $\widetilde{H}[X]$ contains at most two edges (since $|X|\le 3$ and $\widetilde{H}$ is bipartite) this implies
that
\begin{align}\label{eq:tildeH}
2\alpha n  > e(\widetilde{H}) & \ge e(\widetilde{H}-X)+ \left( \sum_{x \in X} d_{\widetilde{H}}(x) \right) -2\ge e(H''-X)-1 +|X|(3 \gamma n/ 5  -2)-2\nonumber \\
& \ge e(H'')-\sum_{x \in X} d_{H''}(x)+|X|(3 \gamma n/ 5  -2)-3\nonumber \\
& =e(H'')-\sum_{x \in X} (d_{H''}(x)-3 \gamma n/ 5  +2)-3
\end{align}
and so
\begin{align}\label{eq:eH''}
e(H'')& \overset{(\ref{dH'})}{< } 2 \alpha n + \sum_{x \in X} \left( \frac{d_{G^{\diamond}[A',B']}(x)+\eps' n}{K^2} - 3 \gamma n /5 +2 \right) +3.
\end{align}
Note that (b$_4$) and (\ref{eq:tildeH}) together imply that if $e(H'')\ge 4\alpha n$ then
$e(\widetilde{H})\ge e(H'')-|X|(31\alpha n/30- 3 \gamma n/ 5  +2)-3\ge 2\alpha n$. Thus $e(H'')< 4\alpha n$ and by (\ref{eH'})
we have $e_{G^{\diamond}}(A',B')\le 4K^2\alpha n/(1-\eps')\le 5K^2\alpha n\le 3n$. Hence
\begin{align}\label{alpha3}
	e_{G^{\diamond}}(A',B') & \stackrel{(\ref{eH'})}{\le}  K^2 e(H'')+\eps' e_{G^{\diamond}}(A',B')\le K^2 e(H'')+3\eps'n\nonumber \\
& \stackrel{(\ref{eq:eH''})}{\le}
  D- \phi n + 7 \eps' n + \sum_{x \in X } \left( d_{G^{\diamond}[A',B']}(x) -  K^2(3\gamma n/5) \right).
\end{align}
Let $G'$ be any subgraph of $G^{\diamond}[A',B']$ which maximises $e(G')$ under the constraint that $\Delta(G') \le K^2(3 \gamma / 5 +2\eps' ) n$.
Note that if $d_{G^{\diamond}[A',B']}(v) \ge  K^2(3 \gamma / 5 +2 \eps' ) n $, then $v\in V_0$ (by~(iv) and (ESch3)) and so
$d_{H''}(v) > 3 \gamma n/ 5$ by~\eqref{dH'}. This in turn implies that $v \in X$.
Hence%
   \COMMENT{The +2 in the next inequality accounts for the edges in $G^{\diamond}[A'\cap X,B'\cap X]$.}
\begin{eqnarray}
	e(G') & \le & e_{G^{\diamond}}(A',B') -  \sum_{x \in X } \left( d_{G^{\diamond}[A',B']}(x) -  K^2( 3\gamma /5+ 2\eps' )n  \right)+2 \nonumber\\
	& \overset{\eqref{alpha3}}{\le} & D- \phi n + 7K^2\eps' n.\label{eq:eG'}
\end{eqnarray}
Note that (\ref{dH'}) together with the fact that $X\neq \emptyset$ implies that
$$\Delta(G[A',B'])\ge \Delta (G^\diamond[A',B'])\ge K^2(3 \gamma n/ 5 -2) -\eps' n\stackrel{(\ref{alpha}), (\ref{alphahier})}{\ge} 11D/40.$$
Since $G$ is not critical this means that
there exists a subgraph $G''$ of $G[A',B']$ such that $\Delta(G'') \le 11 D /40 \le   K^2( 3\gamma /5+ 2 \eps' )n  $ and $e(G'') \ge 41 D/40$.
Thus
\begin{align*}
	D - \phi n + 7K^2 \eps' n \stackrel{(\ref{eq:eG'})}{\ge} e(G') \ge e(G'') - e_{G_0}(A',B') \ge 41 D/40 - \phi n,
\end{align*}
which is a contradiction.
Therefore, we must have $e(\widetilde{H}) \ge 2 \alpha n$.
Hence (b$_6$) is satisfied.
\end{proof}

%%%%%%%%%%%%%%%%%%%%%%%%%%%%%%%%%%%%%%%%

\subsubsection{Step $2$: Decomposing $H''(i,i')$ into Hamilton exceptional system candidates}

Our next aim is to decompose each $H''(i,i')$ into $\alpha n$%
	\COMMENT{$\alpha n = \frac{D- \phi n }{2K^2}$ as defined in Lemma~\ref{lma:move}.}
 Hamilton exceptional system candidates (this will follow from Lemma~\ref{lma:BESdecomprelim}).
Before we can do this, we need the following result on decompositions of bipartite graphs into `even matchings'.
We say that a matching is \emph{even} if it contains an even number of edges, otherwise it is \emph{odd}.

\begin{prop}\label{prop:evenmatching}
Suppose that $0<1/n \ll \gamma \le 1$ and that $n, \gamma n \in \mathbb N$.
Let $H$ be a bipartite graph on $n$ vertices with $\Delta(H) \le 2\gamma n/3$ and where $ e(H) \geq 2 \gamma n$ is even.
Then $H$ can be decomposed into $\gamma n$ edge-disjoint non-empty even matchings, each of size at most $3e(H)/(\gamma n)$.
\end{prop}

\begin{proof}
First note that since $e(H) \ge 2\gamma n $, it suffices to show that $H$ can be decomposed into at most $\gamma n $ edge-disjoint non-empty even matchings,
each of size at most $3e(H)/(\gamma n)$. Indeed, by splitting these matchings further if necessary, one can obtain precisely $\gamma n $
non-empty even matchings.

Set $n' := \lfloor 2\gamma n/3 \rfloor$.
K\"onig's theorem implies that $\chi '(H) \le n'$. So Proposition~\ref{prop:matchingdecomposition} implies that there
is a decomposition of $H$ into $n'$ edge-disjoint
matchings $M_1, \dots ,M_{n'}$ such that $|e(M_s) - e(M_{s'})| \le 1$ for all $s,s' \le n'$.
Hence we have 
\begin{align*}
2 \le \frac{e(H)}{n'} -1 \le e(M_s) \le \frac{e(H)}{n'}+1 \le \frac{3e(H)}{\gamma n}
\end{align*}
for all $s \le n'$.%
   \COMMENT{The final inequality holds since $ \frac{e(H)}{n'}+1\le  \frac{e(H)}{\gamma n/2}+1\le \frac{3e(H)}{\gamma n}$ since $\frac{e(H)}{\gamma n}\ge 2$.}
Since $e(H)$ is even, there are an even number of odd matchings.
Let $M_s$ and $M_{s'}$ be two odd matchings.
So $e(M_s),e(M_{s'}) \ge 3$ and thus there exist two disjoint edges $e \in M_s$ and $e' \in M_{s'}$.
Hence, $M_s - e$, $M_{s'} - e'$ and $\{e, e'\}$ are three even matchings.
Thus, by pairing off the odd matchings and repeating this process,  the proposition follows.
\end{proof}

\begin{lemma} \label{lma:BESdecomprelim}
Suppose that $0 < 1/n \ll \eps_0 \ll \gamma  < 1$, that $ \gamma + \gamma' < 1$ and that $n, \gamma n, \gamma' n \in \mathbb N$.
Let $H$ be a bipartite graph on $n$ vertices with vertex classes $A \dot\cup A_0$ and $ B  \dot\cup B_0$, where $|A_0|+|B_0| \le \eps_0 n$.
Suppose that
\begin{itemize}
	\item[{\rm (i)}] $e(H)$ is even, $\Delta(H) \le 16 \gamma n /15$ and $\Delta(H[A,B]) < (3 \gamma /5 - \epszero) n$. 
\end{itemize}
Let $H'$ be a spanning subgraph of $H$ which maximises $e(H')$ under the constraints  that $\Delta(H') \le 3 \gamma n /5$, $H[A_0,B_0] \subseteq H'$
and $e(H')$ is even. Suppose that
\begin{itemize}
\item[{\rm (ii)}] 
$2( \gamma +  \gamma' )n \le e(H') \le 10 \epszero  \gamma n^2$.	
\end{itemize}
Then there exists a decomposition of $H$ into edge-disjoint Hamilton exceptional system candidates $F_1, \dots ,F_{\gamma n},F'_1, \dots ,F'_{\gamma' n}$
with parameter $\eps_0$ such that $e(F'_s) = 2$ for all $s \le \gamma' n$.
\end{lemma}

Since we are in the non-critical case with many edges between $A'$ and $B'$, we will be able to assume that the subgraph $H'$
satisfies (ii).

Roughly speaking, the idea of the proof of Lemma~\ref{lma:BESdecomprelim} is to apply the previous proposition to decompose $H'$ into a suitable number of even matchings 
$M_i$ (using the fact that it has small maximum degree).% 
\COMMENT{AL: rephrased the following sentences}
We then extend these matchings into Hamilton exceptional system candidates to cover all edges of $H$.
The additional edges added to each $M_i$ will be vertex-disjoint from $M_i$ and form vertex-disjoint 2-paths $uvw$ with $v \in V_0$.
So the number of connections from $A'$ to $B'$ remains the same (as $H$ is bipartite).
Each matching $M_i$ will already be a Hamilton exceptional system candidate, which means that $M_i$ and its extension will have the correct number of connections from $A'$ to $B'$ 
(which makes this part of the argument simpler than in the critical case).

\removelastskip\penalty55\medskip\noindent{\bf Proof of Lemma~\ref{lma:BESdecomprelim}. }
Set $A': = A_0 \cup A$ and $B':= B_0 \cup B$. We first construct the $F'_s$.
If $\gamma '=0$, there is nothing to do. So suppose that $\gamma ' >0$.
Note that each $F'_s$ has to be a matching of size~2 (this follows from the definition of a Hamilton exceptional system candidate and the
fact that $e(F'_s) = 2$).
Since $H'$ is bipartite and so
$$\frac{e(H')}{\chi '(H')}=\frac{e(H')}{\Delta(H')}\ge \frac{2( \gamma +  \gamma' )n}{3 \gamma n /5} > \frac{10}{3},$$
we can find a 2-matching $F'_1$ in $H'$. Delete the edges in $F'_1$ from $H'$ and choose another 2-matching $F'_2$.
We repeat this process until we have chosen $\gamma'n$ edge-disjoint $2$-matchings $F'_1, \dots ,F'_{\gamma'n}$.

We now construct  $F_1, \dots ,F_{\gamma n}$ in two steps: first we construct matchings $M_1, \dots, M_{\gamma n}$ in $H'$
and then extend each $M_i$ into the desired $F_i$.
Let $H_1$ and $H'_1$ be obtained from $H$ and $H'$ by removing all the edges in $F'_1, \dots ,F'_{\gamma'n}$.
So now $2\gamma n \le e(H'_1) \le 10 \epszero \gamma n^2 $ and both $e(H_1)$ and $e(H'_1)$ are even. Thus Proposition~\ref{prop:evenmatching}
implies that there is a decomposition of $H'_1$ into edge-disjoint non-empty even matchings $M_1, \dots ,M_{\gamma n }$,
each of size at most $30 \epszero  n$. 

Note that each $M_i$ is a Hamilton exceptional system candidate with parameter~$\eps_0$.
So if $H'_1 = H_1$, then we are done by setting $F_s: = M_s$ for each $s \le \gamma n$.
Hence, we may assume that $H'' := H_1 - H'_1=H-H'$ contains edges.
Let $X$ be the set of all those vertices $x \in A_0 \cup B_0$ for which $d_{H''}(x)> 0$.
Note that each $x\in X$ satisfies $N_{H''}(x)\subseteq A\cup B$ (since $H[A_0,B_0] \subseteq H'$).
%Together with the fact that $\Delta(H[A,B]) < (3 \gamma /5 - \epszero) n$ by~(i) 
This implies that
each $x \in X$ satisfies
%   \COMMENT{$e(H'')=e(H)-e(H')$ is even and adding any two edges of $H''$ to $H'$ must result in a vertex of degree $>3\gamma n/5$.}
$d_{H'}(x) \ge \lfloor  3\gamma n /5 \rfloor -1$ or $d_{H''}(x)=1$.
(Indeed, suppose that $d_{H'}(x) \le \lfloor  3\gamma n /5 \rfloor -2$ and $d_{H''}(x)\ge 2$.
Then we can move two edges incident to $x$ from $H''$ to $H'$.
The final assumption in (i) and the assumption on $d_{H'}(x)$ together imply that we would still have $\Delta(H') \le 3 \gamma n/5$, a contradiction.)
Since  $\Delta(H) \le 16 \gamma n /15$ by~(i) this in turn implies that
$d_{H''}(x) \le 7 \gamma n /15 +2$ for all $x \in X$.

Let $\mathcal{M}$ be a random subset of $\{M_1, \dots, M_{\gamma n}\}$ where each $M_i$ is chosen independently with probability $2/3$.
By Proposition~\ref{prop:chernoff}, with high probability, the following assertions hold:
\begin{align}
	r : = |\mathcal{M}| & = (2/3 \pm \epszero){\gamma n} \nonumber \\
	|\{M_s \in \mathcal{M} : d_{M_s}(v) = 1 \}| & = 2 d_{H'_1}(v)/3  \pm \epszero{\gamma n}	& \text{for all } v \in V(H). \label{eqn:Mi}
\end{align}
By relabeling if necessary, we may assume that $\mathcal{M} = \{ M_1, M_2, \dots, M_r\}$.
For each $s \le r$, we will now extend $M_s$ to a Hamilton exceptional system candidate $F_s$ with parameter $\eps_0$ by adding edges from~$H''$.
Suppose that for some $1\le s\le r$ we have already constructed $F_1, \dots ,F_{s-1}$. Set $H''_s := H'' - \sum_{j < s} F_j$.
Let $W_s$ be the set of all those vertices $w \in X$ for which $d_{M_s}(w)=0$ and $d_{H''_s}(w) \ge 32\epszero n  \ge 2|A_0\cup B_0|+ e(M_s)$.
Recall that $X\subseteq A_0\cup B_0$ and $N_{H''_s}(w) \subseteq N_{H''}(w) \subseteq A \cup B$ for each $w\in X$ and thus also for
each $w \in W_s$. Thus there are $|W_s|$ vertex-disjoint 2-paths $uwu'$ with $w \in W_s$ and $u,u' \in N_{H''_s}(w) \setminus V(M_s)$.
Assign these 2-paths to $M_s$ and call the resulting graph $F_s$.
Observe that $F_s$ is a Hamilton exceptional system candidate with parameter $\eps_0$.
Therefore, we have constructed $F_1, \dots ,F_r$ by extending $M_1, \dots ,M_r$.

We now construct $F_{r+1}, \dots ,F_{\gamma n}$. For this, we first prove that the above construction implies that the current `leftover' $H_{r+1}''$
has small maximum degree. Indeed, 
note that if $w \in W_s$, then $d_{H''_{s+1}}(w) = d_{H''_{s}}(w) - 2$.
By~\eqref{eqn:Mi}, for each $ x \in X$, the number of $M_s \in \mathcal{M}$ with $d_{M_s}(x) = 0 $ is 
\begin{align}
r - |\{M_s \in \mathcal{M} : d_{M_s}(x) = 1 \}| 
& \ge 	(2/3-\epszero)\gamma n - (2d_{H'_1}(x)/3 + \epszero \gamma n)  \nonumber \\
& \ge 2\gamma n/3- 2d_{H'}(x)/3  -2 \epszero \gamma n  \nonumber \\
& \ge 	2\gamma n/3- 2/3\cdot \lfloor  3\gamma n /5 \rfloor   -2 \epszero \gamma n  \nonumber \\
& \ge  (4/15 - 2 \epszero)\gamma n > d_{H''}(x) /2 . \nonumber
\end{align}
Hence, we have $d_{H''_{r+1}}(x) <  32 \epszero n$ for all  $ x \in X$
(as we remove $2$ edges at $x$ each time we have $d_{M_s}(x)=0$ and $d_{H''_s}(x) \ge 32\epszero n$).
Note that by definition of $H'$, all but at most one edge in $H''$ must have an endpoint in $X$. 
So for $x \notin X$,  $d_{H''}(x) \le |X|+1 \leq |A_0\cup B_0|+1\le \epszero n+1$.
Therefore, $\Delta(H''_{r+1}) <  32 \epszero n$.

Let $H''':=H_1-(F_1+\dots +F_r) $.
So $H'''$ is the union of $H''_{r+1}$ and all the $M_s$ with $r < s \le \gamma n$.
Since each of $H_1$ and $F_1,\dots,F_r $ contains an even number of edges, $e(H''')$ is even.
In addition, $M_s \subseteq H'''$ for each $r< s \le \gamma n$, so $e(H''') \ge 2(\gamma n-r)$.
By~\eqref{eqn:Mi}, since $\Delta(H''_{r+1}) \le 32 \epszero n$, we deduce that for every vertex $v \in V(H''')$,
we have
\begin{align*}
	d_{H'''} (v) & \le \left( \frac{d_{H'_1}(v)}3+ \epszero\gamma n \right)  + \Delta(H''_{r+1})
\le \frac{3\gamma n/5}{3}+\eps_0 \gamma n + 32\epszero n 
\le \frac{2(\gamma n-r)}3
\end{align*}
In the second inequality, we used that $d_{H'_1}(v)\le d_{H'}(v)$.
Moreover, we have
$$e(H''')=e(H''_{r+1})+e(M_{r+1} + \dots +  M_{\gamma n})\le 32\epszero n^2+ 30\eps_0n (\gamma n-r)\le 62\epszero n^2.
$$
Thus, by Proposition~\ref{prop:evenmatching} applied with $H'''$ and $\gamma-r/n$ playing the roles of $H$ and~$\gamma$,
there exists a decomposition of $H'''$ into $\gamma n-r$ edge-disjoint non-empty even matchings $F_{r+1}$, $\dots$, $F_{\gamma n}$,
each of size at most $3e(H''')/(\gamma n-r)\le \sqrt{\eps_0}n/2$. Thus each such $F_s$ is 
a Hamilton exceptional system candidate with parameter $\eps_0$. This completes the proof.
\endproof

\subsubsection{Step $3$: Constructing the localized exceptional systems}

The next lemma will be used to extend most of the exceptional system candidates guaranteed by Lemma~\ref{lma:BESdecomprelim}
into localized exceptional systems. These extensions are required to be `faithful' in the following sense.
Suppose that $F$ is an exceptional system candidate.%
	\COMMENT{Previsouly, $F$ is an $(i,i')$-ESC, but don't think we need it in the definition.}
Then $J$ is a \emph{faithful extension of} $F$ if the following holds:
\begin{itemize}
\item $J$ contains $F$ and $F[A',B']=J[A',B']$.
\item If $F$ is a Hamilton exceptional system candidate, then $J$ is a
Hamilton exceptional system and the analogue holds if $F$ is a matching exceptional system candidate.
\end{itemize}

\begin{lemma} \label{lma:BESextend2}
Suppose that $0 <  1/n  \ll \epszero \ll   1$, that $0 \le \gamma \le 1$ and that $n, K, m, \gamma n \in \mathbb{N}$.
Let $\mathcal{P}$ be a $(K, m , \epszero)$-partition of a set $V$ of $n$ vertices.
Let $1\le i,i'\le K$.
Suppose that $H$ and $F_1, \dots ,F_{\gamma n}$ are pairwise edge-disjoint graphs which satisfy the following conditions:
\begin{itemize}
\item[\rm (i)] $V(H)=V$ and $H$ contains only $A_0A_i$-edges and $B_0B_{i'}$-edges.
\item[\rm (ii)] Each $F_s$ is an $(i,i')$-ESC with parameter~$\eps_0$.
\item[\rm (iii)] Each $v \in V_0$ satisfies $d_{H+\sum F_s}(v )  \ge (2 \gamma  +\sqrt{\epszero} )n$.
\end{itemize}
Then there exist edge-disjoint $(i,i')$-ES $J_1, \dots ,J_{\gamma n }$ with parameter $\eps_0$ in $H +\sum F_s$ 
such that $J_s$ is a faithful extension of $F_s$
for all $s \le \gamma n$. 
\end{lemma}
\begin{proof}
For each $s \le \gamma n$ in turn, we extend $F_s$ into an $(i,i')$-ES $J_s$ with parameter $\eps_0$ in $H+\sum F_s$ such that
$J_s$ and $J_{s'}$ are edge-disjoint for all $s' < s$.
Since $H$ does not contain any $A'B'$-edges, the $J_s$ will automatically satisfy $J_s[A',B'] = F_s[A',B']$.
Suppose that for some $1\le s\le \gamma n$ we have already constructed $J_1, \dots ,J_{s-1}$.
Set $H_s := H - \sum_{s' < s} J_{s'}$. 
Consider any $v \in V_0$.
Since $v$ has degree at most 2 in an exceptional system and in an exceptional system candidate, (iii) implies that
\begin{align*}
	d_{H_s}(v) & \ge d_{H+ \sum F_s}(v) - 2 \gamma n \ge \sqrt{\epszero} n.
\end{align*}
Together with~(i) this shows that condition (ii) in Lemma~\ref{lma:ESextend} holds (with $H_s$ playing the role of $G$).
Since $\cP$ is a $(K,m,\eps_0)$-partition of $V$, Lemma~\ref{lma:ESextend}(i) holds too.
Hence we can apply Lemma~\ref{lma:ESextend} to obtain an exceptional system $J_s$ with parameter $\eps_0$ in $H_s + F_s$
such that $J_s$ is a faithful extension of $F_s$. (i) and (ii) ensure that $J_s$ is an $ (i,i')$-ES, as required.
\end{proof}

\subsubsection{Step $4$: Constructing the remaining exceptional systems.}
Due to condition~(iii), Lemma~\ref{lma:BESextend2} cannot be used to extend \emph{all} the exceptional system candidates
returned by Lemma~\ref{lma:BESdecomprelim} into localized exceptional systems.
The next lemma will be used to deal with the remaining exceptional system candidates
(the resulting exceptional systems will not be localized).

\begin{lemma} \label{globalBES}
Suppose that $0 <  1/n  \ll \epszero \ll \eps' \ll \lambda \ll  1$ and that $n, \lambda n \in \mathbb{N}$.
Let $A,A_0,B,B_0$ be a partition of a set $V$ of $n$ vertices such that $|A_0| +|B_0| \le \epszero n$ and $|A| = |B|$.
Suppose that $H,F_1, \dots ,F_{\lambda n}$ are pairwise edge-disjoint graphs  which satisfy the following conditions:
\begin{itemize}
\item[{\rm (i)}] $V(H)=V$ and $H$ contains only $A_0A$-edges and $B_0B$-edges.
\item[{\rm (ii)}] Each $F_s$ is an exceptional system candidate with parameter $\eps_0$.
\item[{\rm (iii)}] For all but at most $\eps' n$ indices $s\le \lambda n$ the graph $F_s$ is either
a matching exceptional system candidate with $e(F_s) = 0$ or a Hamilton exceptional system candidate with $e(F_s) = 2$.
In particular, all but at most $\eps' n$ of the $F_s$ satisfy $d_{F_s}(v) \le 1$ for all $v \in V_0$.
\item[{\rm (iv)}] All $v \in V_0$ satisfy $d_{H+\sum F_s}(v )  = 2 \lambda n$.
\item[{\rm (v)}] All $v \in A \cup B $ satisfy $d_{H+\sum F_s}(v) \le 2 \epszero n$.
\end{itemize}
Then there exists a decomposition of $H+\sum F_s$ into edge-disjoint exceptional 
systems $J_1, \dots ,J_{\lambda n }$ with parameter $\eps_0$
such that $J_s$ is a faithful extension of $F_s$ for all $s \le \lambda n$. 
\end{lemma}

\begin{proof}
Let $V_0:=A_0 \cup B_0$ and let $v_1, \dots ,v_{|V_0|}$ denote the vertices of $V_0$. 
We will decompose $H$ into graphs $J'_s$ in such a way that
the graphs $J_s:=J'_s + F_s$ satisfy $d_{J_s}(v_i) = 2$ for all $i\le |V_0|$ and $d_{J_s}(v) \le 1$ for all $v \in A \cup B$.
Hence each $J_s$ will be an exceptional system with parameter~$\eps_0$.%
   \COMMENT{Each $J_s$ must be a path system since a cycle would have to lie in $J_s[V_0]$ and so by (i) already in $F_s[V_0]$.}
Condition (i) guarantees that $J_s$ will be a faithful extension of~$F_s$.
Moreover, the $J_s$ will form a decomposition of $H+\sum F_s$.
We construct the decomposition of $H$ by considering each vertex $v_i$ of $A_0 \cup B_0$ in turn.

Initially, we set $V(J'_s)=E(J'_s) = \emptyset$ for all $s \le \lambda n$.
Suppose that for some $1\le i\le |V_0|$ we have already assigned (and added) all the edges of $H$
incident with each of $v_1, \dots, v_{i-1}$ to the $J'_s$.
Consider $v_i$. Without loss of generality assume that $v_i \in A_0$. Note that $N_H(v_i)\subseteq A$ by~(i).
Define an auxiliary bipartite graph $Q_i$ with vertex classes $V_1$ and $V_2$ as follows:
$V_1:=N_{H} (v_i)$ and $V_2$ consists of $2-d_{F_s}(v_i)$ copies of $F_s$ for each $ s\le \lambda n $. 
Moreover, $Q_i$ contains an edge between $v \in V_1$ and $F_s \in V_2$ if and only if $v\notin V(F_s + J'_s)$.%
    \COMMENT{We can write this in this way since by def an exceptional system candidate does not contain
isolated vertices in $A\cup B$. If we change the def, then we here have to write that $v$ is not incident to an
edge of $F_s$.}

We now show that $Q_i$ contains a perfect matching. For this, 
note that $|V_1|=2\lambda n - d_{\sum F_s} (v_i)= |V_2|$ by~(iv). 
(v) implies that for each $v \in V_1 \subseteq A$ we have  $d_{\sum (F_s+J'_s)}(v) \le d_{H+\sum F_s}(v) \le 2 \epszero n$.
So $v$ lies in at most $2 \epszero n$ of the graphs $F_s+J'_s$.
Therefore, $d_{Q_i} (v) \geq |V_2| -4 \eps _0n \geq |V_2|/2$ for all $v \in V_1$. (The final
inequality follows since (iii) and (iv) together imply that
$d_H(v_i)=2\lambda n-d_{\sum F_s} (v_i)\ge 2\lambda n -(\lambda n-\eps' n)-2\eps' n\ge \lambda n/2$%
    \COMMENT{The $2\eps' n$ comes from the fact that $d_{F_s}(v_i)\le 2$ for every exceptional system candidate $F_s$.} 
and so $|V_2|=|V_1|\geq \lambda n/2$.)
On the other hand, since each $F_s+J'_s$ is an exceptional system candidate with parameter $\eps_0$, (ESC3) implies that
$|V(F_s+J'_s)\cap A| \le (\sqrt{\epszero}/2  +2 \epszero) n \le \sqrt{\epszero} n$ for each $F_s \in V_2$.
Therefore $d_{Q_i} (F_s) \geq |V_1| -|V(F_s+J'_s)\cap A| \geq |V_1|/2$ for each $F_s\in V_2$.
Thus we can apply Hall's theorem to find a perfect matching $M$ in $Q_i$.
Whenever $M$ contains an edge between $v$ and $F_s$, we add the edge $v_iv$ to $J'_s$.
This completes the desired assignment of the edges of $H$ at $v_i$ to the $J'_s$.
\end{proof}

\subsubsection{Proof of Lemma~\ref{lma:BESdecom}}

In our proof of Lemma~\ref{lma:BESdecom} we will use the following result, which is a consequence of
Lemmas~\ref{lma:BESextend2} and~\ref{globalBES}. 
Given a suitable set of exceptional system candidates in an exceptional scheme, the lemma extends these into exceptional systems which form a
decomposition of the exceptional scheme.
We prove the lemma in a slightly more general form than needed for the current case, 
as we will also use it 
in the other two cases.

\begin{lemma} \label{BEScons}
Suppose that $0 <  1/n  \ll \epszero \ll \eps\ll \eps' \ll \lambda, 1/K \ll  1$, that $1/(7K^2)\le \alpha < 1/K^2$
and that $n, K,m,\alpha n,\lambda n/K^2  \in \mathbb{N}$.
Let $$\gamma:=\alpha -\frac{\lambda}{K^2}\ \ \ \ \ \ \ \ \text{and} \ \ \ \ \ \ \ \ \gamma':=\frac{\lambda}{K^2}.$$
Suppose that the following conditions hold:
\begin{itemize}
\item[{\rm (i)}] $(G^*,\cP)$ is a $(K,m,\eps_0,\eps)$-exceptional scheme with $|G^*|=n$.
\item[{\rm (ii)}] $G^*$ is the edge-disjoint union of $H(i,i')$, $F_1(i,i'),\dots,F_{\gamma n}(i,i')$ and $F'_1(i,i'),\dots,F'_{\gamma' n}(i,i')$
over all $1\le i,i'\le K$. 
\item[{\rm (iii)}] Each $H(i,i')$ contains only $A_0A_i$-edges and $B_0B_{i'}$-edges.
\item[{\rm (iv)}] Each $F_s(i,i')$ is an $(i,i')$-ESC with parameter $\eps_0$.
\item[{\rm (v)}] Each $F'_s(i,i')$ is an exceptional system candidate with parameter $\eps_0$.
Moreover, for all but at most $\eps' n$ indices $s\le \gamma' n$ the graph $F'_s(i,i')$ is either a
matching exceptional system candidate with $e(F'_s(i,i'))=0$ or a Hamilton exceptional system candidate with $e(F'_s(i,i'))=2$.
\item[{\rm (vi)}] $d_{G^*}(v)=2K^2\alpha n$ for all $v\in V_0$.
\item[{\rm (vii)}] For all $1\le i,i'\le K$ let $G^*(i,i'):=H(i,i')+\sum_{s\le \gamma n} F_s(i,i')+\sum_{s\le \gamma' n} F'_s(i,i')$.
Then $d_{G^*(i,i')}(v)=(2\alpha \pm \eps')n$ for all $v\in V_0$.
\end{itemize}
Then $G^*$ has a decomposition into $K^2\alpha n$ edge-disjoint exceptional systems
$$J_1(i,i'),\dots,J_{\gamma n}(i,i') \ \ \ \ \ \ \text{and} \ \ \ \ \ \ J'_1(i,i'),\dots,J'_{\gamma' n}(i,i')$$
with parameter $\eps_0$, where $1\le i,i'\le K$, such that $J_s(i,i')$ is an $(i,i')$-ES which is a
faithful extension of $F_s(i,i')$ for all $s \le \gamma n$
and $J'_s(i,i')$ is a faithful extension of $F'_s(i,i')$ for all $s \le \gamma' n$.%
\COMMENT{Added $J'_s(i,i')$ is a faithful extension of $F'_s(i,i')$ instead of $F'_s(i,i') \subseteq J'_s(i,i')$.}
\end{lemma}
\begin{proof}
Fix any $i,i'\le K$ and set $H:=H(i,i')$ and $F_s:=F_s(i,i')$ for all $s\le \gamma n$.
Our first aim is to apply Lemma~\ref{lma:BESextend2} in order to extend each of $F_1, \dots , F_{\gamma n}$
into a $(i,i')$-HES. (iii) and (iv) ensure that conditions~(i) and~(ii) of Lemma~\ref{lma:BESextend2} hold. To verify Lemma~\ref{lma:BESextend2}(iii),
note that by (v) and (vii) each $v\in V_0$ satisfies
\begin{align*}
d_{H + \sum F_s}(v) & = d_{G^*(i,i')}(v)-d_{\sum_s F'_s(i,i')}(v)\ge (2\alpha-\eps')n-(\gamma'-\eps')n-2\eps' n\\
& =(2\alpha-\gamma'-2\eps') n\ge (2\gamma+\sqrt{\eps_0})n.
\end{align*}
(Here the first inequality follows since (v) implies that $d_{F'_s(i,i')}(v)\le 1$ for all but at most $\eps' n$ indices $s\le \gamma' n$.)
Thus we can indeed apply Lemma~\ref{lma:BESextend2} to find edge-disjoint $(i,i')$-ES $J_1(i,i'), \dots ,J_{\gamma n}(i,i')$ with parameter $\eps_0$
in $H + \sum F_s$ such that $J_s(i,i')$ is a faithful extension of $F_s$ for all $s \le \gamma n$.
We repeat this procedure for all $1\le i,i' \le K$ to obtain $K^2\gamma n$
edge-disjoint (localized) exceptional systems. 

Our next aim is to apply Lemma~\ref{globalBES} in order to construct the $J'_s(i,i')$.
Let $H_0$ be the union of $H(i,i')-(J_1(i,i')+\dots+J_{\gamma n}(i,i'))$ over all $i,i'\le K$.
Relabel the $F'_s(i,i')$ (for all $s\le \gamma' n$ and all $i,i'\le K$) to obtain exceptional system candidates $F'_1,\dots,F'_{\lambda n}$.
Note that by~(vi) each $v\in V_0$ satisfies
\begin{equation}\label{eq:degH0}
d_{H_0+\sum F'_s}(v) =d_{G^*}(v)-2K^2 \gamma n= 2K^2\alpha n-2K^2 \gamma n=2\lambda n.
\end{equation}
Thus condition~(iv) of Lemma~\ref{globalBES} holds with $H_0, F'_s$ playing the roles of $H,F_s$.
(iii) and (v) imply that conditions (i)--(iii) of Lemma~\ref{globalBES} hold with $K^2 \eps'$ playing the role of $\eps'$.
To verify Lemma~\ref{globalBES}(v), note that
each $v\in A$ satisfies $d_{H_0+\sum F'_s}(v)\le d_{G^*}(v,A_0)+d_{G^*}(v,B')\le 2\eps_0 n$ by~(iii), (i) and (ESch3).
Similarly each $v\in B$ satisfies $d_{H_0+\sum F'_s}(v)\le 2\eps_0 n$. Thus we can apply
Lemma~\ref{globalBES} with $H_0, F'_s, K^2\eps'$ playing the roles of $H, F_s, \eps'$ to obtain
a decomposition of $H_0+\sum_s F'_s$ into
$\lambda n$ edge-disjoint exceptional systems $J'_1,\dots,J'_{\lambda n}$ with parameter $\eps_0$ such that
$J'_s$ is a faithful extension of $F'_s$ for all $s\le \lambda n$.
Recall that each $F'_s$ is a $F'_{s'}(i,i')$ for some $i,i'\le K$ and some $s'\le \gamma' n$.
Let $J'_{s'}(i,i'):=J'_s$. Then all the $J_s(i,i')$ and all the $J'_s(i,i')$ are as required in the lemma. 
\end{proof}

We will now combine Lemmas~\ref{lma:move},~\ref{lma:BESdecomprelim} and~\ref{BEScons} in order
to prove Lemma~\ref{lma:BESdecom}.

\removelastskip\penalty55\medskip\noindent{\bf Proof of Lemma~\ref{lma:BESdecom}. }
Let $G^{\diamond}$ be as defined in Lemma~\ref{lma:BESdecom}(iv).
Choose a new constant $\eps'$ such that $\eps \ll \eps' \ll \lambda, 1/K$.
Set 
\begin{align} \label{alpha'}
2 \alpha n & := \frac{ D - \phi n }{K^2}, 
& \gamma_1 & : = \alpha -  \frac{2\lambda}{K^2}
& \text{and} &
& \gamma'_1 & : = \frac{2\lambda }{K^2}.
\end{align}
Similarly as in the proof of Lemma~\ref{lma:move}, since $\phi  \ll 1/3\le D/n$, we have
\begin{align} \label{alphahier2}
	\alpha \ge 1/(7K^2), \ \ \ \ \ (1-14\lambda)\alpha\le \gamma_1 <\alpha \ \ \ \ \  \text{and} \ \ \ \ \ \eps \ll \eps'  \ll \lambda, 1/K,\alpha, \gamma_1 \ll 1.
\end{align}
Apply Lemma~\ref{lma:move} with $\gamma_1$ playing the role of $\gamma$ in order to obtain a decomposition
of $G^{\diamond}$ into edge-disjoint spanning subgraphs $H(i,i')$ and $H''(i,i')$ (for all $1\le i,i' \le K$) which satisfy the following
properties, where $G'(i,i'):=H(i,i')+H''(i,i')$:
\begin{itemize}
\item[\rm (b$_1$)] Each $H(i,i')$ contains only $A_0A_i$-edges and $B_0B_{i'}$-edges.
\item[\rm (b$_2$)] $H''(i,i')\subseteq G^{\diamond}[A',B']$. Moreover,
all but at most $\eps' n$ edges of $H''(i,i')$ lie in $G^{\diamond}[A_0 \cup A_i, B_0 \cup B_{i'}]$.
\item[\rm (b$_3$)] $e(H''(i,i'))$ is even and $ 2 \alpha n \le e(H''(i,i')) \le 11\epszero n^2/(10K^2)$.
\item[\rm (b$_4$)] $\Delta(H''(i,i')) \le 31 \alpha n/30 $.
\item[\rm (b$_5$)] $d_{G'(i,i')}(v )  =  \left( 2 \alpha \pm   \eps' \right) n $ for all $v \in V_0$.
\item[\rm (b$_6$)] Let $\widetilde{H}$ any spanning subgraph of $H''(i,i')$ which maximises $e(\widetilde{H})$
under the constraints that $\Delta(\widetilde{H}) \le 3\gamma_1 n /5$, $H''(i,i')[A_0,B_0] \subseteq \widetilde{H}$ and $e(\widetilde{H})$ is even.
Then $e(\widetilde{H}) \ge 2 \alpha n $.
\end{itemize}
Fix any $1 \le i,i' \le K$.%
\COMMENT{AL: added $1 \le$}
 Set $H := H(i,i')$ and $H'' := H''(i,i')$. Our next aim is to 
decompose $H''$ into suitable `localized' Hamilton exceptional system candidates. For this, we will
apply 
Lemma~\ref{lma:BESdecomprelim} with $H'', \gamma_1,\gamma'_1$ playing the roles of $H, \gamma,\gamma'$.
Note that $\Delta(H'') \le 31 \alpha n/30 \le 16\gamma_1 n/15$ by (b$_4$) and (\ref{alphahier2}).
Moreover, $\Delta(H''[A,B])\le \Delta(G^\diamond[A,B])\le \eps_0 n$ by~(iv) and (ESch3). Since $e(H'')$ is even by (b$_3$), it follows that
condition~(i) of Lemma~\ref{lma:BESdecomprelim} holds. Condition~(ii) of Lemma~\ref{lma:BESdecomprelim}
follows from (b$_6$) and the fact that any $\widetilde{H}$ as in (b$_6$) satisfies
$e(\widetilde{H})\le e(H'') \le 11\eps_0 n^2/(10K^2)\le 10\eps_0\gamma_1 n^2$
(the last inequality follows from (\ref{alphahier2})). Thus we can indeed apply Lemma~\ref{lma:BESdecomprelim}
in order to decompose $H''$ into $ \alpha n $ edge-disjoint Hamilton exceptional system candidates $F_1, \dots , F_{\gamma_1 n}, F'_1, \dots , F'_{\gamma_1 ' n}$
with parameter $\eps_0$ such that $e(F'_s) =2$ for all $s \le \gamma_1' n$. 
Next we set
\begin{align*}
\gamma_2 & := \alpha - \frac{\lambda}{K^2} & 
& \text{and} &
\gamma'_2 & : = \frac{\lambda}{K^2}.
\end{align*}
Condition (b$_2$) ensures that by relabeling the $F_s$'s and $F'_s$'s we obtain $ \alpha n $ edge-disjoint Hamilton exceptional system
candidates $F_1(i,i'), \dots , F_{\gamma_2 n}(i,i'), F'_1(i,i'), \dots , F'_{\gamma_2 ' n}(i,i')$ with parameter $\eps_0$%
     \COMMENT{Note here that $\gamma_1$ and $\gamma_1'$ are replaced with $\gamma_2$ and $\gamma_2'$ respectively.}
such that properties~(a$'$) and (b$'$) hold:
\begin{itemize}
	\item[(a$'$)] $F_s(i,i')$ is an $(i,i')$-HESC for every $s \le \gamma_2 n$. Moreover, at least $\gamma_2 ' n$ of the $F_s(i,i')$ satisfy $e(F_s(i,i')) = 2$.
	\item[(b$'$)] $e(F'_s(i,i')) = 2$ for all but at most $ \eps' n$ of the $F'_s(i,i')$.
\end{itemize}
Indeed, we can achieve this by relabeling each $F_s$ which is a subgraph of $G^{\diamond}[A_0 \cup A_i, B_0 \cup B_{i'}]$
as one of the $F_{s'}(i,i')$ and each $F_s$ for which is not the case as one of the $F'_{s'}(i,i')$.

Our next aim is to apply Lemma~\ref{BEScons} with $G^\diamond,\gamma_2,\gamma_2'$ playing the roles of $G^*,\gamma,\gamma'$.
Clearly conditions (i) and (ii) of Lemma~\ref{BEScons} hold. (iii) follows from (b$_1$). (iv) and (v) follow from (a$'$) and (b$'$).
(vi) follows from Lemma~\ref{lma:BESdecom}(i),(iii). Finally, (vii) follows from (b$_5$) since $G'(i,i')$ plays the role of $G^*(i,i')$.
Thus we can indeed apply Lemma~\ref{BEScons} to obtain a decomposition of $G^\diamond$ into $K^2\alpha n$
edge-disjoint Hamilton exceptional systems $J_1(i,i'),\dots,J_{\gamma_2 n}(i,i')$ and $J'_1(i,i'),\dots,J'_{\gamma'_2 n}(i,i')$
with parameter $\eps_0$, where $1\le i,i'\le K$, such that $J_s(i,i')$ is an $(i,i')$-HES which is a faithful extension of  $F_s(i,i')$ for all $s \le \gamma_2 n$
and $J'_s(i,i')$ is a faithful extension of $F'_s(i,i')$ for all $s \le \gamma'_2 n$.
Then the set $\mathcal{J}$ of all these Hamilton exceptional systems is as required in Lemma~\ref{lma:BESdecom}.
\endproof

%%%%%%%%%%%%%%%%%%%%%%%%%%%%%%%%%%%%%%%%%%%%%%%%%%%%%%%%%%%%%%%%%%%%%%%%%%%%%%%%%%%%%%%%%%%%%%%%%%%%%%%%%%

\subsection{Critical case with $e(A',B') \ge D$} \label{sec:critical}

The aim of this section is to prove the following analogue of Lemma~\ref{lma:BESdecom} for the case when $G$ is critical and $e_G(A',B') \ge D$.
For this, recall that $G$ is critical if $\Delta(G[A',B']) \ge 11 D/40$ and $e(H) \le 41 D/40$ for all subgraphs $H$ of $G[A',B']$ such that $\Delta(H) \le 11 D/40$.
By Lemma~\ref{critical}(ii) we know that in this case $D = (n-1)/2$ or $D = n/2-1$.

\begin{lemma}\label{lma:BESdecomcritical}
Suppose that $0 <  1/n  \ll \epszero \ll  \eps \ll   \lambda, 1/K \ll 1$, that $D\ge n - 2\lfloor n/4 \rfloor -1$,%
    \COMMENT{Daniela: added the bound on $D$, which we need to apply Lemma~\ref{critical}}
that $0\le \phi\ll 1$ and that
$n, K, m, \lambda n/K^2, (D - \phi n)/(400K^2) \in \mathbb{N}$.%
    \COMMENT{Previously also had that $\phi n \in \mathbb{N}$. But this follows since $D,K, (D - \phi n)/(400K^2) \in \mathbb{N}$.}
Suppose that the following conditions hold:
\begin{itemize}
	\item[{\rm(i)}] $G$ is a $D$-regular graph on $n$ vertices.
	\item[{\rm(ii)}] $\mathcal{P}$ is a $(K, m, \epszero)$-partition of $V(G)$ such that%
\COMMENT{Previously this also included
$d_{G[A',B']}(v)\le \eps_0 n \text{ for all } v\in A\cup B.$ However, we can replaced this by the bound on $d_{G^\diamond[A',B']}(v)$ implied by~(iv).} 
$e_G(A',B') \ge D$ and $\Delta(G[A',B']) \le D/2$.
Furthermore, $G$ is critical. In particular, $e_G(A',B')< n$ and $D = (n-1)/2$ or $D= n/2-1$ by Lemma~\ref{critical}(ii) and~(iii).%
	\COMMENT{Also $n = 1 \pmod4$ or $n = 0 \pmod4$ respectively.}
	\item[{\rm(iii)}] $G_0$ is a subgraph of $G$ such that $G[A_0]+G[B_0] \subseteq G_0$, $e_{G_0}(A',B') \leq \phi n$ and $d_{G_0}(v) = \phi n $ for all $v \in V_0$.
	\item[{\rm(iv)}] Let $G^{\diamond} := G - G[A] - G[B] -G_0$. $e_{G^\diamond} (A',B')$ is even and
$(G^{\diamond}, \mathcal{P})$ is a $(K, m, \epszero,\eps)$-exceptional scheme.
	\item[{\rm(v)}] Let $w_1$ and $w_2$ be (fixed) vertices such that $d_{G[A',B']}(w_1) \ge d_{G[A',B']}(w_2) \ge d_{G[A',B']}(v)$
for all $v \in V(G) \setminus\{w_1,w_2\}$.
Suppose that
\begin{equation} \label{additional}
d_{G^{\diamond}[A',B']}(w_1), d_{G^{\diamond}[A',B']}(w_2) \le (D-\phi n)/2.
\end{equation}
\end{itemize}
Then there exists a set $\mathcal{J}$ consisting of $(D-\phi n)/2$ edge-disjoint Hamilton exceptional systems with parameter $\eps_0$ in $G^{\diamond}$
which satisfies the following properties: 
\begin{itemize}
    \item[\rm (a)] Together the Hamilton exceptional systems in $\mathcal{J}$ cover all edges of $G^\diamond$.
	\item[\rm (b)] For each $1\le i,i' \le K$, the set $\mathcal{J}$ contains $(D - (\phi + 2 \lambda) n)/(2K^2)$ $(i,i')$-HES.
	Moreover, $\lambda n /K^2$ of these $(i,i')$-HES are such that
\begin{itemize}
\item[\rm (b$_1$)] $e_J(A',B') =2$ and 
\item[\rm (b$_2$)] $d_{J[A',B']}(w) = 1$ for all $w \in \{w_1,w_2\}$ with $d_{G[A',B']}(w) \ge 11D/40$.%
    \COMMENT{i.e. $w \in W' \cap \{w_1,w_2\}$}
\end{itemize}
\end{itemize}
\end{lemma}

Similarly as for Lemma~\ref{lma:BESdecom}, (b) implies that $\mathcal{J}$ contains $\lambda n$ Hamilton exceptional systems
which might not be localized. Another similarity is that when constructing the robustly decomposable graph $G^{\rm rob}$ in~\cite{paper1}, 
we only use those $J_s$ which have some additional useful properties, namely~(b$_1$) and~(b$_2$) in this case. 
This gives us a way of satisfying~(\ref{additional}) in the second application of
Lemma~\ref{lma:BESdecomcritical} in~\cite{paper1} (i.e.~after the removal of $G^{\rm rob}$), by `tracking' the degrees of the high
degree vertices $w_1$ and $w_2$.
Indeed, if  $d_{G[A',B']}(w_2)\ge 11D/40$, then (b$_2$) will imply that $d_{G^{\rm rob}[A',B']}(w_i)$ is large for $i=1,2$.
This in turn means that after removing $G^{\rm rob}$, in the leftover graph $G^\diamond$, $d_{G^\diamond[A',B']}(w_i)$ is comparatively small, 
i.e.~condition~(\ref{additional}) will hold in the  second application  of Lemma~\ref{lma:BESdecomcritical}.

Condition~(\ref{additional}) itself is natural for the following reason:
suppose for example that it is violated  for $w_1$ and that $w_1 \in A_0$. 
Then for some Hamilton exceptional system $J$ returned by the lemma, both edges of $J$ incident to $w_1$ will
have their other endpoint in $B'$. So (the edges at) $w_1$ cannot be used as a `connection' between $A'$ and $B'$
in the Hamilton cycle which will extend $J$, and it may be impossible to find such a connection elsewhere.

The overall strategy for the proof of Lemma~\ref{lma:BESdecomcritical} is similar to that of Lemma~\ref{lma:BESdecom}.
As before, it consists of four steps. 
In Step~1, we use Lemma~\ref{lma:movecritical} instead of Lemma~\ref{lma:move}. In Step~2, we use Lemma~\ref{lma:BESdecomprelim2} instead of
Lemma~\ref{lma:BESdecomprelim}. We still use Lemma~\ref{BEScons} which combines Steps~3 and~4.

%%%%%%%%%%%%%%%%%%%%%%%%%%%%%%%%%%%%%

\subsubsection{Step $1$: Constructing the graphs $H''(i,i')$}
The next lemma is an analogue of Lemma~\ref{lma:move}. We will apply it with the graph $G^\diamond$ from Lemma~\ref{lma:BESdecomcritical}(iv)
playing the role of $G$. Note that instead of assuming that our graph $G$ given in Lemma~\ref{lma:BESdecomcritical} is critical, the
lemma assumes that $e_{G^{\diamond}}(A',B')\le 2n$. This is a weaker assumption, since if $G$ is critical, then
$e_{G^\diamond}(A',B')\le e_G(A',B') < n$ by Lemma~\ref{critical}(iii). 
Using only this weaker assumption has the advantage that we can also apply the lemma in the proof of Lemma~\ref{lma:PBESdecom}, i.e.~the
case when $e_G(A',B') < D$. (b$_7$) is only used in the latter application.

\begin{lemma} \label{lma:movecritical}
Suppose that $0 <  1/n  \ll \epszero \ll  \eps \ll 1/K \ll 1$ and that $n, K,m\in \mathbb{N}$.
Let $(G, \mathcal{P})$ be a $(K, m, \epszero, \eps )$-exceptional scheme with $|G| = n$ and $e_G(A_0), e_G(B_0) =0$.
Let $W_0$ be a subset of $V_0$ of size at most~$2$ such that for each $w \in W_0$, we have%
    \COMMENT{In general, $W_0 = \{ w_1, w_2 \}$, where $w_1$ and $w_2$ are the two vertices such that $d_{G[A',B']}(w_1) \ge  d_{G[A',B']}(w_2) \ge d_{G[A',B']}(v)$ for $v \in V(G)$.
By Lemma~\ref{lma:BESdecomcritical}(v), $d_{G^{\diamond}[A',B']}(w_1), d_{G^{\diamond}[A',B']}(w_2) \le (D- \phi n)/2$.
Recall that $e_G(A',B') \ge D$ and $e_{G_0}(A',B') \le \phi n$, so we have $d_{G^{\diamond}[A',B']}(w) \le e_{G^{\diamond}}(A',B')/2$.}
\begin{equation}\label{eq:degw}
K^2 \le d_{G[A',B']}(w) \le e_G(A',B')/2.
\end{equation}
Suppose that $e_G(A',B')\le 2n$ is even.
Then $G$ can be decomposed into edge-disjoint spanning subgraphs $H(i,i')$ and $H''(i,i')$ of $G$ (for all $1 \le i,i' \le K$)%
\COMMENT{AL: added $1 \le$}
such that the following properties hold, where $G'(i,i'):=H(i,i')+H''(i,i')$:
\begin{itemize}
\item[\rm (b$_1$)] Each $H(i,i')$ contains only $A_0A_i$-edges and $B_0B_{i'}$-edges.
\item[\rm (b$_2$)] $H''(i,i')\subseteq G[A',B']$. Moreover, all but at most  $20 \eps n/K^2$ edges of $H''(i,i')$ lie in $G[A_0 \cup A_i, B_0 \cup B_{i'}]$.
\item[\rm (b$_3$)] $e(H''(i,i')) =  2 \left\lceil e_G(A',B') / (2 K^2) \right\rceil$ or $e(H''(i,i')) =  2 \left\lfloor e_G(A',B') / (2 K^2) \right\rfloor$.
\item[\rm (b$_4$)] $d_{H''(i,i')}(v )  =  ( d_{G[A',B']}(v)  \pm 25 \eps n)/K^2$ for all $v \in V_0$.
\item[\rm (b$_5$)] $d_{G'(i,i')}(v )  =  \left( d_{G}(v) \pm   25 \eps n \right)  / K^2 $ for all $v \in V_0$.
\item[\rm (b$_6$)] Each $w \in W_0$ satisfies $d_{H''(i,i')}(w) = \lceil d_{G[A',B']}(w) / K^2 \rceil $ or 
$d_{H''(i,i')}(w) = \lfloor d_{G[A',B']}(w) / K^2 \rfloor $.
\item[\rm (b$_7$)] Each $w \in W_0$ satisfies $2 d_{H''(i,i')}(w) \le  e(H''(i,i'))$.%
    \COMMENT{This statement is used when $e_G(A',B') < D$.}
\end{itemize}
\end{lemma}
\begin{proof}
Since $e_G(A',B')$ is even, there exist unique non-negative integers $b$ and $q$ such that $e_G(A',B') = 2K^2 b+ 2q$ and $q < K^2$.
Hence, for all $1 \le i,i' \le K$,%
\COMMENT{AL: changed to  $1 \le i,i' \le K$}
 there are integers $b_{i,i'}\in \{2b, 2b+2\}$ such that $\sum_{i,i' \le K} b_{i,i'} = e_G(A',B')$.
In particular, the number of pairs $i,i'$ for which $b_{i,i'}=b+2$ is precisely~$q$.
We will choose the graphs $H''(i,i')$ such that $e(H''(i,i'))=b_{i,i'}$. (In particular, this will ensure that (b$_3$) holds.)
The following claim will help to ensure~(b$_6$) and~(b$_7$). 

\medskip

\noindent
\textbf{Claim.} \emph{For each $w \in W_0$ and all $i,i' \le K$ there is an integer  $a_{i,i'}=a_{i,i'}(w)$ which satisfies the following properties:
\begin{itemize}
	\item $a_{i,i'} = \lceil d_{G[A',B']}(w) /K^2 \rceil $ or $a_{i,i'} = \lfloor d_{G[A',B']}(w) /K^2 \rfloor$.
	\item $2a_{i,i'} \le b_{i,i'}$.
	\item $\sum_{i,i' \le K} a_{i,i'} = d_{G[A',B']} (w)$.
	\end{itemize}}
	
\smallskip

\noindent To prove the claim,
note that there are unique non-negative integers $a$ and $p$ such that $d_{G[A',B']} (w) = K^2 a + p $ and $p<K^2$.
Note that $a\ge 1$ by (\ref{eq:degw}). Moreover,
\begin{align}\label{eq:apbq}
 2(K^2 a +p) =   2 d_{G[A',B']} (w) \stackrel{(\ref{eq:degw})}{\le} e_G(A',B') = 2 K^2 b+ 2q.
\end{align}
This implies that $a\le b$. Recall that $b_{i,i'} \in \{2b,2b+2\}$. So if $b > a$, then the claim holds by
choosing any $a_{i,i'}\in \{a,a+1\}$ such that $\sum_{i,i' \le K} a_{i,i'} = d_{G[A',B']} (w)$.
Hence we may assume that $a = b$. Then~(\ref{eq:apbq}) implies that $p \le q$.
Therefore, the claim holds by setting $a_{i,i'}:= a+1$ for exactly $p$ pairs $i,i'$ for which $b_{i,i'} = 2b+2$ and
setting $a_{i,i'}:= a$ otherwise. This completes the proof of the claim.

\medskip

\noindent 
Apply Lemma~\ref{lma:randomslice} to decompose $G$ into subgraphs $H(i,i')$, $H'(i,i')$ (for all $i,i' \le K$) satisfying the following properties,
where $G(i,i') = H(i,i') + H'(i,i')$:
\begin{itemize}
\item[\rm (a$'_1$)] Each $H(i,i')$ contains only $A_0A_i$-edges and $B_0B_{i'}$-edges.
\item[\rm (a$'_2$)] All edges of $H'(i,i')$ lie in $G[A_0 \cup A_i, B_0 \cup B_{i'}]$.
\item[\rm (a$'_3$)] $e ( H'(i,i') )   =  ( e_{G}(A',B')  \pm  8 \eps n ) /K^2$.
\item[\rm (a$'_4$)] $d_{H'(i,i')}(v )  =  ( d_{G[A',B']}(v)  \pm 2 \eps n)/K^2$ for all $v \in V_0$.
\item[\rm (a$'_5$)] $d_{G(i,i')}(v )  =  ( d_{G}(v)  \pm 4 \eps n)/K^2$ for all $v \in V_0$.
\end{itemize}
Indeed, (a$'_3$) follows from Lemma~\ref{lma:randomslice}(a$_3$) and our assumption that $e_G(A',B') \le 2n$.

Clearly, (a$'_1$) implies that the graphs $H(i,i')$ satisfy (b$_1$). We will now move some $A'B'$-edges of $G$ between
the $H'(i,i')$ such that the graphs $H''(i,i')$ obtained in this way satisfy the following conditions:
\begin{itemize}
\item Each $H''(i,i')$ is obtained from $H'(i,i')$ by adding or removing at most $20\eps n/K^2$ edges of $G$.
\item $e(H''(i,i'))=b_{i,i'}$.
\item $d_{H''(i,i')}(w)=a_{i,i'}(w)$ for each $w\in W_0$, where $a_{i,i'}(w)$ are integers satisfying the claim.
\end{itemize}
Write $W_0=:\{w_1\}$ if $|W_0|=1$ and $W_0=:\{w_1,w_2\}$ if $|W_0|=2$.
If $W_0 \ne \emptyset$, then (a$'_4$) implies that $d_{H'(i,i')}(w_1 )  =   a_{i,i'}(w_1)  \pm (2 \eps n/K^2+1)$.
For each $i,i'\le K$, we add or remove at most $2 \eps n /K^2+1$ edges incident to $w_1$
such that the graphs $H''(i,i')$ obtained in this way satisfy $d_{H''(i,i')}(w_1) =  a_{i,i'}(w_1)$.
Note that since $a_{i,i'}(w_1) \ge \lfloor d_{G[A',B']}(w_1) /K^2 \rfloor \ge 1$ by~(\ref{eq:degw}), we can do this in such a way that
we do not move the edge $w_1w_2$ (if it exists).%
\COMMENT{This is not really needed.}
Similarly, if $|W_0|=2$, then for each $i,i'\le K$ we add or remove at most $2 \eps n /K^2+1$ edges incident to $w_2$
such that the graphs $H''(i,i')$ obtained in this way satisfy $d_{H''(i,i')}(w_2) =  a_{i,i'}(w_2)$.
As before, we do this in such a way that we do not move the edge $w_1w_2$ (if it exists).

Thus $d_{H''(i,i')}(w_1)  =  a_{i,i'}(w_1)$ and $d_{H''(i,i')}(w_2) =  a_{i,i'}(w_2)$
for all $1 \le i,i' \le K$%
\COMMENT{AL: added $1 \le$}
 (if $w_1,w_2$ exist). In particular, together with the claim this implies
that $d_{H''(i,i')}(w_1) ,d_{H''(i,i')}(w_2) \le b_{i,i'}/2$. Thus the number of edges
of $H''(i,i')$ incident to $W_0$ is at most
\begin{align}
	\sum_{w \in W_0}d_{H''(i,i')}(w) \le b_{i,i'}.  \label{H''1}
\end{align}
(This holds regardless of the size of $W_0$.) On the other hand, (a$'_3$) implies that for all $i, i' \le K$ we have
\begin{align}
e ( H''(i,i') )   =  ( e_{G}(A',B')  \pm  8 \eps n ) /K^2 \pm 2(2 \eps n /K^2 + 1) = b_{i,i'} \pm 13 \eps n/K^2. \nonumber
\end{align}
Together with~\eqref{H''1} this ensures that we can add or delete at most $13 \eps n/K^2$ edges which do not intersect $W_0$
to or from each $H''(i,i')$ in order to ensure that $e(H''(i,i')) = b_{i,i'}$ for all $i,i' \le K$.
Hence, (b$_3$), (b$_6$) and (b$_7$) hold. Moreover,
\begin{equation}\label{eq:edgediff}
e(H''(i,i')-H'(i,i'))\le |W_0| (2\eps n/K^2 +1) + 13 \eps n/K^2 \le 20 \eps n/K^2.
\end{equation}
So (b$_2$) follows from (a$'_2$). Finally, (b$_4$) and (b$_5$) follow from (\ref{eq:edgediff}), (a$'_4$) and (a$'_5$).
\end{proof}

%%%%%%%%%%%%%%%%%%%%%%%%%%%%%%%%%%%%%%

\subsubsection{Step $2$: Decomposing $H''(i,i')$ into Hamilton exceptional system candidates}

Before we can prove an analogue of Lemma~\ref{lma:BESdecomprelim}, we need the following result.
It will allow us to distribute the edges incident to the (up to three) vertices $w_i$ of high degree in $G[A',B']$ in a suitable way
among the localized Hamilton exceptional system candidates $F_j$.
The degrees of these high degree vertices $w_i$ will play the role of the $a_i$.
The $c_j$ will account for edges (not incident to $w_i$) which have already been assigned to the $F_j$.
(b) and (c) will be used to ensure  (ESC4), i.e.~that the total number of `connections' between $A'$ and $B'$ is even and positive.

\begin{lemma} \label{matrix}
Let $1 \le q \le 3$ and $0 \le \eta < 1$ and $r, \eta r \in \mathbb{N}$.
Suppose that $a_1, \dots, a_q \in \mathbb{N}$ and $c_1, \dots, c_r \in \{0,1,2\}$ satisfy the following conditions:
\begin{itemize}
	\item[\rm (i)] $c_1 \ge \dots \ge c_r \ge c_1-1 $.
	\item[\rm (ii)] $\sum_{i \le q} a_i + \sum_{j \le r} c_j = 2(1+\eta)r$.
	\item[\rm (iii)] $31 r /60  \le a_1, a_2 \le r$ and $31 r /60  \le a_3 \le 31r/30$.
\end{itemize}
Then for all $i \le q$ and all $j \le r$ there are $a_{i,j} \in \{0,1,2\}$ such that the following properties hold:
\begin{itemize}
	\item[\rm (a)] $\sum_{j \le r} a_{i,j} = a_i$ for all $i \le q$.
	\item[\rm (b)] $c_j + \sum_{i \le q} a_{i,j} = 4$ for all $j \le \eta r$ and $c_j + \sum_{i \le q} a_{i,j} = 2$ for all $\eta r < j \le r$.
	\item[\rm (c)] For all $j \le r$ there are at least $2- c_j$ indices $i \le q$ with $a_{i,j}=1$.
\end{itemize}
\end{lemma}
\begin{proof}
We will choose $a_{i,1}, \dots, a_{i,r}$ for each $i\le q$ in turn such that the following properties
($\alpha_i$)--($\rho_i$) hold, where we write $c_j^{(i)} : = c_j + \sum_{i' \le i} a_{i',j}$ for each $ 0 \le  i \le q$ (so $c_j^{(0)}=c_j$):
\begin{itemize}
\item[\rm ($\alpha_i$)] If $i \ge 1$ then $\sum_{j \le r} a_{i,j} = a_i$.
\item[\rm ($\beta_i$)] $4\ge c^{(i)}_1 \ge \dots \ge c^{(i)}_r$.
\item[\rm ($\gamma_i$)] If $\sum_{j\le r} c^{(i)}_j < 2r$, then $| c^{(i)}_j - c^{(i)}_{j'} | \le 1$ for all $j,j' \le r$.
\item[\rm ($\delta_i$)] If $\sum_{j\le r} c^{(i)}_j \ge 2r$, then $c^{(i)}_j \ge 2$ for all $j \le \eta r$ and $c^{(i)}_j = 2$ for all $\eta r< j \le r$.
\item[\rm ($\rho_i$)] If $1\le i\le q$ and $c^{(i-1)}_j < 2$ for some $j\le r$, then $a_{i,j}  \in \{ 0 , 1 \}$.
\end{itemize}
We will then show that the $a_{i,j}$ defined in this way are as required in the lemma.

Note that (i) and the fact that $c_1, \dots, c_r \in \{0,1,2\}$ together imply ($\beta_0$)--($\delta_0$).
Moreover, ($\alpha_0$) and ($\rho_0$) are vacuously true.
Suppose that for some $1\le i\le q$ we have already defined $a_{i',j}$ for all $i'<i$ and all $j\le r$ such that ($\alpha_{i'}$)--($\rho_{i'}$) hold.
In order to define $a_{i,j}$ for all $j\le r$, we distinguish the following cases.

\medskip 

\noindent\textbf{Case 1: $\sum_{j \le r} c^{(i-1)}_j \ge 2r$.}

\smallskip

\noindent Recall that in this case $c^{(i-1)}_j \ge 2$ for all $j \le r$ by ($\delta_{i-1}$).
For each $j\le r$ in turn we choose $a_{i,j}\in \{0,1,2\}$ as large as possible subject to the constraints
that
\begin{itemize}
\item $a_{i,j}+ c^{(i-1)}_j\le 4$ and
\item $\sum_{j'\le j} a_{i,j'}\le a_i$.
\end{itemize}
Since $c^{(i)}_j=a_{i,j}+ c^{(i-1)}_j$, ($\beta_i$) follows from ($\beta_{i-1}$) and our choice of the $a_{i,j}$.
($\gamma_i$) is vacuously true. To verify ($\delta_i$), note that $c_j^{(i)} \ge c_j^{(i-1)}\ge 2$ by ($\delta_{i-1}$).
Suppose that the second part of ($\delta_{i}$) does not hold, i.e.~that $c_{\eta n+1}^{(i)}>2$.
This means that $a_{i,\eta n+1}>0$.
Together with our choice of the $a_{i,j}$ this implies that $c_j^{(i)}=4$ for all $j\le \eta n$.
Thus%
   \COMMENT{We can't replace $\le$ by $=$ since at the moment we only know that $\sum_{j\le r} a_{i,j}\le a_i$.}
\begin{align*}
2(1+\eta) r & =4\eta r+2(r-\eta r)< \sum_{j\le r} c^{(i)}_j=\sum_{j\le r} a_{i,j}+\sum_{i'<i} a_{i'}+ \sum_{j\le r} c_j
\le \sum_{i'\le i} a_{i'}+ \sum_{j\le r} c_j
\end{align*}
contradicting~(ii). Thus the second part of ($\delta_{i}$) holds too. Moreover, 
$c_{\eta n+1}^{(i)}=c_{\eta n+1}^{(i-1)}=2$ also means that $a_{i,\eta n+1}=0$.
So $\sum_{j'\le \eta n} a_{i,j'}= a_i$, i.e.~ ($\alpha_i$) holds. ($\rho_i$) is vacuously true since $c^{(i-1)}_j \ge 2$ by ($\delta_{i-1}$).

\medskip

\noindent\textbf{Case 2: $2r - a_i  \le \sum_{j \le r} c^{(i-1)}_j < 2r$.}

\smallskip

\noindent
If $i\in \{1,2\}$ then together with (iii) this implies that
\begin{equation}\label{eq:sum1st}
\sum_{j \le r} c^{(i-1)}_j\ge r\ge a_i.
\end{equation}
If $i=3$ then
\begin{equation}\label{eq:sum2nd}
\sum_{j \le r} c^{(i-1)}_j\ge \sum_{j\le r} \sum_{i' \le 2} a_{i',j}=a_1+a_2\ge \frac{31r}{30}\ge a_3
\end{equation}
 by~(iii). In particular, in both cases we have $\sum_{j \le r} c^{(i-1)}_j\ge r$.
Together with ($\gamma_{i-1}$) this implies that $c^{(i-1)}_j \in \{ 1 , 2\}$ for all $j \le r$.
Let $0\le r'\le r$ be the largest integer such that $c^{(i-1)}_{r'} = 2$. 
So $r' < r$ and $\sum_{j \le r} c^{(i-1)}_j  = r+r' $. Together with (\ref{eq:sum1st}) and (\ref{eq:sum2nd})
this in turn implies that $a_i \le r+r'$ (regardless of the value of~$i$). 

Set $a_{i,j}:=1$ for all $r'<j\le r$. Note that
$$\sum_{r'<j\le r} a_{i,j}=r-r'=2r-\sum_{j \le r} c^{(i-1)}_j \le a_i,$$
where the final inequality comes from the assumption of Case~2. Take $a_{i,1},\dots,a_{i,r'}$ to be a sequence of the form $2,\dots,2,0,\dots,0$
(in the case when $a_i-\sum_{r'<j\le r} a_{i,j}$ is even)
or $2,\dots,2,1,0,\dots,0$ (in the case when $a_i-\sum_{r'<j\le r} a_{i,j}$ is odd) which is chosen in such a way
that $\sum_{j\le r'} a_{i,j}=a_i-\sum_{r'<j\le r} a_{i,j} = a_i-r+r'$. This can be done since $a_i \le r+r'$ implies that the right hand side is at most $2r'$.

Clearly, ($\alpha_i$), ($\beta_i$) and ($\rho_i$) hold.
Since $\sum_{j\le r} c^{(i)}_j=a_i+\sum_{j\le r} c^{(i-1)}_j \ge 2r$ as we are in Case~2,
($\gamma_i$) is vacuously true. Clearly, our choice of the $a_{i,j}$ guarantees that $c^{(i)}_j\ge 2$ for all $j\le r$.
As in Case~1 one can show that%
    \COMMENT{If $c^{(i)}_{\eta r+1}> 2$ then we must have that $r'>\eta r$ and $c^{(i)}_j= 4$ for all $j\le \eta r$.
Now the same argument as in Case~1 gives a contradiction.}
$c^{(i)}_j= 2$ for all $\eta r<j\le r$. Thus ($\delta_i$) holds.

\medskip

\noindent\textbf{Case 3: $\sum_{j \le r} c^{(i-1)}_j < 2r - a_i$.}

\smallskip

\noindent
Note that in this case
$$
2r>\sum_{j \le r} c^{(i-1)}_j +a_i = \sum_{i' \le i} a_{i'} + \sum_{j \le r }c_j ,$$
and so $i < q$ by (ii). Together with (iii) this implies that $a_i \le r$. Thus for all $j \le r$
we can choose $a_{i,j} \in \{0,1\}$ such that ($\alpha_i$)--($\gamma_i$) and ($\rho_i$) are satisfied.
($\delta_i$) is vacuously true.

\medskip

This completes the proof of the existence of numbers $a_{i,j}$ (for all $i\le q$ and all $j\le r$)
satisfying ($\alpha_i$)--($\rho_i$). It remains to show that these $a_{i,j}$ are as required in the lemma.
Clearly, ($\alpha_1$)--($\alpha_q$) imply that (a) holds. Since
$c_j^{(q)}  = c_j + \sum_{i \le q} a_{i,j}$ the second part of (b) follows from ($\delta_q$).
Since $c_j^{(q)}\le 4$ for each $j\le \eta r$ by~($\beta_q$), together with (ii) this in turn implies that the
first part of (b) must hold too. If $c_j < 2$, then ($\rho_1$)--($\rho_q$) and (b) together imply that for at least
$2-c_j$ indices $i$ we have $a_{i,j} = 1$. Therefore, (c) holds.
\end{proof}

%%%%%%%%%%%%%%%%%%%%%%%%%%%%%%%%%%%%%%%%%%%%%%%%%%%%%%%%%%%%%%%%%%%%%%%%%%%%%%%%%%%%%%%%%%%%%%%%%%%%%%%%%%

We can now use the previous lemma to decompose the bipartite graph induced by $A'$ and $B'$ into 
Hamilton exceptional system candidates. 

\begin{lemma} \label{lma:BESdecomprelim2}
Suppose that $0< 1/n \ll \eps_0 \ll \alpha < 1$, that $0\le  \eta < 199/200$ and
that%
   \COMMENT{Previously had $\eta \alpha n/200\in\mathbb{N}$. But in the proof of Lemma~\ref{lma:BESdecomcritical}
we cannot ensure that this holds, which is fine since we do not need it.}
$n, \alpha n/200, \eta \alpha n  \in \mathbb{N}$.
Let $H$ be a bipartite graph on $n$ vertices with vertex classes $A \dot\cup A_0$ and $ B  \dot\cup B_0$ where $|A_0|+|B_0| \le \eps_0 n$.
Furthermore, suppose that the following conditions hold:
\begin{itemize}
	\item[\rm (c$_1$)] $e(H) = 2(1+ \eta)\alpha n$.
	\item[\rm (c$_2$)] There is a set $W'\subseteq V(H)$ with $1 \le |W'| \le 3$%
\COMMENT{$W'$ will be the vertex set as defined in Lemma~\ref{critical}.} 
	and such that 
\begin{align*}
\textrm{$e(H- W' ) \le 199\alpha n/100$ and $d_H(w) \ge 13 \alpha n / 25 $ for all $w \in W'$.}
\end{align*}
	\item[\rm (c$_3$)] There exists a set $W_0 \subseteq W'$ with $|W_0| = \min \{2, |W'| \}$ and such that $d_H(w)  \le \alpha n$ for all $w \in W_0$
and $d_H(w') \le 41 \alpha n / 40$ for all $w' \in W' \setminus W_0$.%
	\COMMENT{Later on, we will take $W_0 = \{w_1, w_2\}$, where $w_1$ and $w_2$ are the two vertices that we can bound the degree in $G^{\diamond}[A',B']$.}
	\item[\rm (c$_4$)] For all $w \in W'$ and all $v \in V(H) \setminus W'$ we have $d_H(w) - d_H(v) \ge \alpha n/150$.
	\item[\rm (c$_5$)] For all $v \in A \cup B$ we have $d_H(v) \le \epszero n$. 
\end{itemize}
Then there exists a decomposition of $H$ into edge-disjoint Hamilton exceptional system candidates $F_1, \dots, F_{\alpha n }$
such that $e(F_s) = 4$ for all $s \le \eta \alpha n$ and $e(F_s) = 2$ for all $\eta \alpha n < s \le \alpha n $.
Furthermore, at least $\alpha n/200$ of the $F_s$ satisfy the following two properties:
\begin{itemize}
\item $d_{F_s}(w) =1 $ for all $w \in W_0$,
\item $e(F_s) = 2$.
\end{itemize}
\end{lemma}

Roughly speaking, the idea of the proof is first to find the $F_s$ which satisfy the final two properties.
Let $H_1$ be the graph obtained from $H$ by removing the edges in all these $F_s$. We will
decompose $H_1-W'$ into matchings $M_j$ of size at most two.
Next, we extend these matchings into Hamilton exceptional system candidates $F_j$ using Lemma~\ref{matrix}.
In particular, if $e(M_j)<2$, then we will use one or more edges incident to $W'$ to ensure that
the number of $A'B'$-connections is positive and even, as required by (ESC4).
(Note that it does not suffice to ensure that the number of $A'B'$-edges is positive and even for this.)

\begin{proof}
Set $H':=H - W'$, $W_0 =: \{w_1, w_{|W_0|} \}$ and $W' =: \{ w_1, \dots, w_{|W'|} \}$.
Hence, if $|W'| = 3$, then $W' \setminus W_0 = \{ w_3 \}$. Otherwise $W'=W_0$.

We will first construct $e_{H}(W')$ Hamilton exceptional system candidates $F_s$, such that each of them is a matching
of size two and together they cover all edges in $H[W']$. So suppose that $e_{H}(W')>0$. Thus $|W'| =2$ or $|W'| = 3$. 
If $|W'| = 2$, let $f$ denote the unique edge in $H[W']$. Note that
$$e(H')  \ge e(H) - (d_H(w_1) + d_H(w_2) -1) \ge 2(1+\eta)\alpha n-(2\alpha n-1)\ge 1$$
by (c$_1$) and~(c$_3$). So there exists an edge $f'$ in $H'$. Therefore, $M'_1 : = \{f,f'\}$ is a matching.
If $|W'| = 3$, then $e_{H}(W') \le 2$ as $H$ is bipartite. Since by (c$_2$) each $w \in W'$ satisfies $d_H(w) \ge 13\alpha n / 25$, it
is easy to construct $e_H(W')$ $2$-matchings $M'_1,M'_{e_{H}(W')}$ such that $d_{M_s'}(w) = 1$ for all $w \in W'$ and all $s \le e_H(W')$
and such that $H[W'] \subseteq M'_1 \cup M'_{e_{H}(W')}$. Set $F_{ \alpha  n - s+1} := M'_s$ for all $s \le e_H(W')$ (regardless of the size of $W'$).

We now greedily choose $\alpha n/200-e_{H}(W')$ additional $2$-matchings $F_{199\alpha n/200+1},\dots,F_{ \alpha n - e_{H}(W')}$ in $H$ which are
edge-disjoint from each other and from $F_{ \alpha n },F_{ \alpha n - e_{H}(W') +1}$ and such that $d_{F_s}(w)=1$
for all $w\in W_0$ and all $199\alpha n/200<s\le  \alpha n - e_{H}(W')$. To see that this can be done, recall that by (c$_2$) we have
$d_H(w)\ge 13 \alpha n / 25$ for all $w\in W'$ (and thus for all $w\in W_0$) and that (c$_1$) and (c$_3$) together imply that
$e(H-W_0)\ge 2(1+ \eta)\alpha n-\alpha n>\alpha n$ if $|W_0|=1$.

Thus $F_{199\alpha n/200+1},\dots,F_{\alpha n}$ are Hamilton exceptional system candidates
satisfying the two properties in the `furthermore part' of the lemma.
Let $H_1$ and $H'_1$ be the graphs obtained from $H$ and $H'$ by deleting all the $\alpha n/100$ edges in these Hamilton exceptional system candidates.
Set
\begin{align}\label{eq:eta'}
r : = 199\alpha n/200 \ \ \ \ \ \  \text{and} \ \ \ \ \ \ \eta' := \eta \alpha n/r=200\eta/199.
\end{align}
Thus $0\le \eta'<1$ and we now have%
   \COMMENT{Note that $e(H)-\alpha n/100=2 ( 1 + \eta )\alpha n -\alpha n/100=2(199/200+\eta )\alpha n=2 ( 1 + \eta' )199\alpha n/200$.}
\begin{align}
	H_1[W'] & = \emptyset, 
	& e(H_1) & = e(H)-\alpha n/100=2 ( 1 + \eta' )r 
	& \text{and} &
	& e(H'_1) & \le 2 r.
	\label{eqn:H}
\end{align}
(To verify the last inequality note that $e(H'_1) \le e(H-W')\le 2 r$ by~(c$_2$).)
Also, (c$_2$) and (c$_4$) together imply that for all $w \in W'$ and all $v \in V(H) \setminus W'$ we have
\begin{align}
d_{H_1}(w) & \ge \alpha n/2\ge 4 \epszero n  & \text{and} & &d_{H_1}(w) - d_{H_1}(v) &\ge 2 \epszero n . \label{eqn:H2}
\end{align}
Moreover, by (c$_2$) and (c$_3$), each $w \in W_0$ satisfies
\begin{align}\label{eq:W'H1}
	31r/60 & \le 13\alpha n/25- \alpha n/200\le d_{H}(w)-d_{H-H_1}(w)=  d_{H_1}(w)\nonumber \\
 & \le \alpha n-\alpha n/200= r .
\end{align}
Similarly, if $|W'|=3$ and so $w_3$ exists, then
\begin{align}\label{eq:W'H1w3}
	31r/60 & \le 13\alpha n/25- \alpha n/200 \le d_{H}(w_3)-d_{H-H_1}(w_3)=d_{H_1}(w_3)\nonumber \\
& \le 41\alpha n/40\le 31r/30.
\end{align}
(\ref{eqn:H2}) and~(\ref{eq:W'H1}) together imply that $d_{H'_1}(v) \le d_{H_1}(v)< d_{H_1}(w_1)\le r $ for all $v\in V(H)\setminus W'$.
Thus $\chi'(H'_1)\le \Delta(H'_1)\le r$.
Together with Proposition~\ref{prop:matchingdecomposition} this implies that $H'_1$ can be decomposed into $r$ edge-disjoint matchings 
$M_1, \dots ,M_{r}$
such that $| m_j- m_{j'}| \le 1$ for all $1 \le j , j' \le r$, where we set $m_j := e(M_j)$.

Our next aim is to apply Lemma~\ref{matrix} with $|W'|$, $d_{H_1}(w_i)$, $m_j$, $\eta'$ playing the roles of $q$, $a_i$, $c_j$, $\eta$
(for all $i\le |W'|$ and all $j \le r$). Since $\sum_{j\le r} m_j=e(H'_1)\le 2r$ by (\ref{eqn:H}) and since
$| m_j- m_{j'}| \le 1$, it follows that $m_j\in \{0,1,2\}$ for all $j\le r$.
Moreover, by relabeling the matchings $M_j$ if necessary, we may assume that $m_1 \ge m_2 \ge \dots \ge m_r$. Thus
condition~(i) of Lemma~\ref{matrix} holds. (ii) holds too since $\sum_{i\le |W'|} d_{H_1}(w_i)+\sum_{j\le r} m_j=e(H_1)=2(1+\eta')r$
by~(\ref{eqn:H}). Finally, (iii) follows from (\ref{eq:W'H1}) and~(\ref{eq:W'H1w3}).
Thus we can indeed apply Lemma~\ref{matrix} in order to obtain numbers $a_{i,j}  \in \{0,1,2\}$ (for all $i \le |W'|$ and $ j \le r $)
which satisfy the following properties:
\begin{itemize}
\item[(a$'$)]	$\sum_{j \le r } a_{i,j} = d_{H_1}(w_i)$ for all $i\le |W'|$.
	\item[(b$'$)]	$ m_j + \sum_{i \le |W'|} a_{i,j} = 4 $ for all $j \le \eta' r$ and
	         $m_j + \sum_{i \le |W'|} a_{i,j} = 2$ for all $\eta' r < j \le  r$.
\item[(c$'$)] 	If $m_j < 2$ then there exist at least $2 - m_j$ indices $i$ such that $a_{i,j} = 1$.
\end{itemize}
For all $j\le r$, our Hamilton exceptional system candidate $F_j$ will consist of the edges in $M_j$
as well as of $a_{i,j}$ edges of $H_1$ incident to $w_i$ (for each $i\le |W'|$).   
So let $F_j^0: = M_j$ for all $j \le r$. For each $i =1,\dots, |W'|$ in turn, we will now assign the edges of $H_1$ incident with $w_i$
to $F_1^{i-1},\dots,F_r^{i-1}$ such that the resulting graphs $F_1^{i},\dots,F_r^{i}$ satisfy the following properties: 
\begin{itemize}
	\item[($\alpha_i$)] If $i \ge 1$, then $e(F_j^i) - e(F_j^{i-1}) = a_{i,j}$.
	\item[($\beta_i$)] $F^i_j$ is a path system. Every vertex $v \in A \cup B$ is incident to at most one edge of $F_j^i$.
For every $v \in V_0 \setminus W'$ we have $d_{F_j^i}(v) \le 2$. If $e(F_j^i) \le 2$, we even have $d_{F_j^i}(v) \le 1$.
    \item[($\gamma_i$)] Let $b_{j}^i$ be the number of vertex-disjoint maximal paths in $F_j^i$ with one endpoint in $A'$ and the other in $B'$.
If $a_{i,j}=1$ and $i \ge 1$, then $b_j^i = b_j^{i-1}+1$. Otherwise $b_j^i = b_j^{i-1}$. 
\end{itemize}
We assign the edges of $H_1$ incident with $w_i$ to $F_1^{i-1},\dots,F_r^{i-1}$ in two steps.
In the first step, for each index $j\le r$ with $a_{i,j}=2$ in turn, we assign an edge of $H_1$ between $w_i$ and $V_0$ to
$F_j^{i-1}$ whenever there is such an edge left. More formally, to do this, we set $N_0 := N_{H_1}(w_i)$.
For each $j \le r$ in turn, if $a_{i,j} = 2$ and $N_{j-1} \cap V_0 \ne \emptyset$, then we choose a vertex
$v \in N_{j-1} \cap V_0$ and set $F_j':= F_j^{i-1} + w_i v$, $N_j: = N_{j-1} \setminus \{v\}$ and $a'_{i,j} := 1$.
Otherwise, we set $F_j':= F_j^{i-1}$, $N_j:= N_{j-1}$ and $a'_{i,j} := a_{i,j}$.

Therefore, after having dealt with all indices $j\le r$ in this way, we have that
\begin{align}
\textrm{either $a'_{i,j} \le 1$ for all $j \le r$ or $ N_r \cap V_0 = \emptyset$ (or both).} \label{V0cond}
\end{align}
Note that by (b$'$) we have $e(F_j') \le m_j + \sum_{i'\le i} a_{i',j} \le 4$ for all $j\le r$.
Moreover, (a$'$) implies that $|N_r| = \sum_{j \le r}  a'_{i,j}$.
Also, $N_r  \setminus V_0=N_{H_1}(w_i) \setminus V_0$, and so $N_{H_1}(w_i) \setminus N_r  \subseteq V_0$.
Hence
\begin{align}
|N_r| & = 
| N_{H_1}(w_i) | - | N_{H_1}(w_i) \setminus N_r | 
\ge d_{H_1}(w_i) - |V_0| \ge d_{H_1}(w_i) - \epszero n. 
\label{v'1}
\end{align}

In the second step, we assign the remaining edges of $H_1$ incident with $w_i$ to $F'_1,\dots,F'_r$.
We achieve this by finding a perfect matching $M$ in a suitable auxiliary graph.

\medskip

\noindent
{\bf Claim.} \emph{Define a graph $Q$ with vertex classes $N_r$ and $V'$ as follows:
$V'$ consists of $a'_{i,j}$ copies of $F_j'$ for each $ j \le r$.
$Q$ contains an edge between $v \in N_r$ and $F'_j \in V'$ if and only $v$ is not an endpoint of an edge in $F'_j$.
Then $Q$ has a perfect matching~$M$.}

\smallskip

\noindent
To prove the claim, note that 
\begin{align}\label{eqn:V1=V2}
	|V'| = \sum_{j \le r} a'_{i,j} = |N_r|  \overset{\eqref{v'1}}{\ge} d_{H_1}(w_i) - \epszero n.  
\end{align}
Moreover, since $F'_j\subseteq H$ is bipartite and so every edge of $F'_j$ has at most one endpoint in $N_r$,
it follows that
\begin{align}
d_{Q} (F'_j) \geq |N_r| - e(F_j')  \ge |N_r| -4 \label{eqn:dQF_j}
\end{align}
for each $F_j' \in V'$.
Consider any $v \in N_r$. Clearly, there are at most $d_{H_1}(v)$ indices $j\le r$ such that $v$ is an endpoint of an
edge of $F_j'$. If $v \in N_r \setminus V_0 \subseteq A \cup B$, then by (c$_5$), $v$ lies in at most
$2 d_{H_1}(v) \le 2 d_{H}(v)\le 2 \epszero n$ elements of $V'$. (The factor~2 accounts for the fact that each $F'_j$ occurs in
$V'$ precisely $a'_{i,j}\le 2$ times.) So
$$
d_{Q} (v) \geq |V'| - 2\epszero n  \overset{\eqref{eqn:V1=V2}}{\geq} d_{H_1}(w_i)- 3\epszero  n  \overset{\eqref{eqn:H2}}{\geq} \epszero n.
$$
If $v \in N_r \cap V_0$, then \eqref{V0cond} implies that $a'_{i,j} \le  1$ for all $j\le r$.
Thus
\begin{align*}
d_{Q} (v) \geq |V'| - d_{H_1}(v) \overset{\eqref{eqn:V1=V2}}{\geq} ( d_{H_1}(w_i) - d_{H_1}(v) )- \epszero n
\overset{\eqref{eqn:H2}}{\geq} 2 \epszero n  - \epszero n = \epszero n .
\end{align*}
To summarize, for all $v \in N_r$ we have $d_{Q} (v) \geq \epszero n $.
Together with \eqref{eqn:dQF_j} and the fact that $|N_r| = |V'|$ by~\eqref{eqn:V1=V2}
this implies that $Q$ contains a perfect matching $M$ by Hall's theorem.
This proves the claim.

\medskip

For each $j\le r$, let $F^i_j$ be the graph obtained from $F'_j$ by adding the edge $w_iv$ whenever the perfect matching~$M$ 
(as guaranteed by the claim) contains
an edge between $v$ and $F'_j$.

Let us now verify ($\alpha_i$)--($\gamma_i$) for all $i \le |W'|$. Clearly, ($\alpha_0$)--($\gamma_0$) hold and $b_j^0=m_j$.
Now suppose that $i \ge 1$ and that ($\alpha_{i-1}$)--($\gamma_{i-1}$) hold.
Clearly, ($\alpha_i$) holds by our construction of $F^i_1,\dots,F^i_r$. 
Now consider any $j\le r$. If $a_{i,j} = 0$, then ($\beta_i$) and ($\gamma_i$) follow from ($\beta_{i-1}$) and ($\gamma_{i-1}$).
If $a_{i,j} = 1$, then the unique edge in $F^i_j-F_j^{i-1}$ is vertex-disjoint from any edge of $F_j^{i-1}$ (by the definition of $Q$)
and so ($\beta_i$) holds.
Moreover, $b_j^i = b_j^{i-1}+1$ and so ($\gamma_i$) holds. 
So suppose that $a_{i,j} = 2$. Then the unique two edges in $F_j^{i} - F_j^{i-1}$ form a path $P=v'w_iv''$ of length two with internal vertex~$w_i$.
Moreover, at least one of the edges of $P$, $w_iv''$ say, was added to $F_j^{i-1}$ in the second step of our construction of $F^i_j$.
Thus $d_{F^i_j}(v'')=1$. The other edge $w_iv'$ of $P$ was either added in the first or in the second step.
If $w_iv'$ was added in the second step, then $d_{F^i_j}(v')=1$. Altogether this shows that in this case
($\gamma_i$) holds and ($\beta_i$) follows from ($\beta_{i-1}$).
So suppose that $w_iv'$ was added to $F_j^{i-1}$ in the first step of our construction of $F^i_j$.
Thus $v'\in V_0\setminus W'$. But since $a_{i,j} = 2$, (b$'$) implies that $e(F_{j}^{i-1}) = m_j + \sum_{i' < i } a_{i',j} \le 2$.
Together with ($\beta_{i-1}$) this shows that $d_{F_{j}^{i-1}}(v) \le 1$ for all $v \in V_0 \setminus W'$.
Hence $d_{F_{j}^{i-1}}(v') \le 1$ and so $d_{F_{j}^{i}}(v') \le 2$. Together with ($\beta_{i-1}$) this implies ($\beta_i$).
(Note that if $e(F_j^{i-1})=0$, then the above argument actually shows that $d_{F_j^i(v')} \le 1$, as required.)
Moreover, the above observations also guarantee that ($\gamma_i$) holds. Thus $F^i_1,\dots,F^i_r$ satisfy ($\alpha_i$)--($\gamma_i$).

After having assigned the edges of $H_1$ incident with $w_i$ for all $i\le |W'|$, we have obtained graphs
$F^{|W'|}_1,\dots,F^{|W'|}_r$. Let $F_j:=F^{|W'|}_j$ for all $j\le r$. Note that by ($\gamma_{|W'|}$) for all $j\le r$
the number of vertex-disjoint maximal $A'B'$-paths in $F_j$ is precisely $b^{|W'|}_j$.

We now claim that $b_j^{|W'|}$ is positive and even.
To verify this,  recall that $b_j^0=m_j$. Let ${\rm odd}_j$ be the number of $a_{i,j}$ with $a_{i,j}=1$ and $i \le |W'|$.
So $b_j^{|W'|}=m_j+{\rm odd}_j$.
Together with (c$'$) this immediately implies that $b_j^{|W'|} \ge 2$.
Moreover,  since $a_{i,j} \in \{0,1,2 \}$ we have
$$
b_j^{|W'|}=m_j+{\rm odd}_j=m_j+\sum_{i \le |W'|,\ a_{i,j} {\rm \ is\ odd}} a_{i,j}.
$$
Together with (b$'$) this now implies that $b_j^{|W'|}$ is even. This proves the claim.

Together with (a$'$), (b$'$) and ($\alpha_i$), ($\beta_i$)
for all $i\le |W'|$ this in turn shows that $F_1,\dots,F_r$ form a decomposition of $H_1$
into edge-disjoint Hamilton exceptional system candidates with
$e(F_j)=4$ for all $j\le \eta' r$ and $e(F_j)=2$ for all $\eta' r<j\le r$. Recall that $\eta' r=\eta \alpha n$
by~(\ref{eq:eta'}) and that we have already constructed Hamilton exceptional system candidates $F_{199\alpha n/200+1},\dots,F_{ \alpha n}$
which satisfy the `furthermore statement' of the lemma, and thus in particular consist of precisely two edges.
This completes the proof of the lemma.
\end{proof}

\subsubsection{Proof of Lemma~\ref{lma:BESdecomcritical}}

We will now combine Lemmas~\ref{lma:movecritical},~\ref{lma:BESdecomprelim2} and~\ref{BEScons} in order
to prove Lemma~\ref{lma:BESdecomcritical}.
This will complete the construction of the required exceptional sequences in the case when $G$ is both critical and $e(G[A',B'])\ge D$.

\removelastskip\penalty55\medskip\noindent{\bf Proof of Lemma~\ref{lma:BESdecomcritical}. }
Let $G^{\diamond}$ be as defined in Lemma~\ref{lma:BESdecomcritical}(iv).
Our first aim is to decompose $G^{\diamond}$ into suitable `localized' subgraphs via Lemma~\ref{lma:movecritical}.
Choose a new constant $\eps'$ such that $\eps \ll \eps' \ll \lambda, 1/K$ and
define $\alpha$ by
\begin{equation} \label{alphaeqD}
2 \alpha n  :=\frac{ D - \phi n}{K^2}.
\end{equation}
Recall from Lemma~\ref{lma:BESdecomcritical}(ii) that $D = (n-1)/2$ or $D = n/2-1$.
Together with our assumption that $\phi \ll 1$ this implies that%
   \COMMENT{$\alpha= \frac{ D/n - \phi}{2K^2}\ge \frac{ 1/2-1/n - \phi}{2K^2}=\frac{ 1-2/n - 2\phi}{4K^2}$.}
\begin{equation}\label{alphahier3}
\frac{ 1-2/n - 2\phi}{4K^2}\le \alpha \le  \frac{ 1 - 2\phi}{4K^2} \ \ \ \  \ \ \ \ \text{and} \ \ \ \  \ \ \ \ \eps \ll \eps'  \ll \lambda, 1/K, \alpha \ll 1.
\end{equation}
Note that by Lemma~\ref{lma:BESdecomcritical}(ii) and (iii) we have $ e_{G^{\diamond}}(A',B') \ge D - \phi n =  2 K^2 \alpha n$.
Together with Lemma~\ref{critical}(iii) this implies that
\begin{align}
	2 K^2 \alpha n \le  e_{G^{\diamond}}(A',B') \le e_G(A',B') \le 17D/10+5
\stackrel{(\ref{alphaeqD})}{\le} {18 K^2 \alpha n}/{5} \stackrel{(\ref{alphahier3})}{<} n . \label{eW'}
\end{align}
Moreover, recall that by Lemma~\ref{lma:BESdecomcritical}(i) and (iii) we have
\begin{equation}\label{eq:degvV0}
d_{G^{\diamond}}(v) = 2K^2 \alpha n \ \ \ \  \ \ \text{for all } v \in V_0.
\end{equation}
Let $W$ be the set of all those vertices $w\in V(G)$ with $d_{G[A',B']}(w)\ge 11D/40$. So $W$ is as defined in Lemma~\ref{critical}
and $1\le |W|\le 3$ by Lemma~\ref{critical}(i). Let $W'\subseteq V(G)$ be as guaranteed by Lemma~\ref{critical}(v). Thus $W\subseteq W'$, $|W'|\le 3$,
\begin{align}\label{eq:degrees}
d_{G[A',B']}(w') & \ge \frac{21D}{80}, & 
d_{G[A',B']}(v) & \le \frac{11D}{40} &
{\rm and} & &
d_{G[A',B']}(w') - d_{G[A',B']}(v)& \ge \frac{D}{240}.
\end{align}
for all $w'\in W'$ and all $v\in V(G)\setminus W'$. In particular, $W'\subseteq V_0$. (This follows since
Lemma~\ref{lma:BESdecomcritical}(iii),(iv) and (ESch3) together imply that
$d_{G[A',B']}(v) = d_{G^\diamond [A',B']}(v)+d_{G_0[A',B']}(v)\le \eps_0 n+e_{G_0}(A',B')\le \eps_0 n+\phi n$ for all $v\in A\cup B$.)
Let $w_1, w_2, w_3$ be vertices of $G$ such that
$$d_{G[A',B']}(w_1) \ge d_{G[A',B']}(w_{2}) \ge d_{G[A',B']}(w_{3}) \ge d_{G[A',B']}(v)$$ for
all $v \in V(G) \setminus \{w_1,w_2,w_3\}$, where $w_1$ and $w_2$ are as in Lemma~\ref{lma:BESdecomcritical}(v).
Hence $W$ consists of $w_1, \dots, w_{|W|}$ and $W'$ consists of $w_1, \dots, w_{|W'|}$.
Set $W_0 : = \{w_1,w_2\} \cap W'$. Since $d_{G_0}(v)=\phi n$ for each $v\in V_0$ (and thus for each $v\in W_0$), each $w \in W_0$ satisfies%
\COMMENT{Previously, we had $15 K^2 \alpha n /38 \stackrel{(\ref{alphaeqD})}{\le} 21D/80-\phi n $, but the first inequality can be replaced by $K^2$,
as this is what we actually use in Lemma~\ref{lma:movecritical}}
\begin{align}
\label{degW0} 
K^2 {\le} 21D/80-\phi n \stackrel{(\ref{eq:degrees})}{\le}d_{G^{\diamond}[A',B']}(w) \le K^2 \alpha n \stackrel{(\ref{eW'})}{\le} e_{G^{\diamond}}(A',B')/2.
\end{align}
(Here the third inequality follows from Lemma~\ref{lma:BESdecomcritical}(v).)
Apply Lemma~\ref{lma:movecritical} to $G^{\diamond}$ in order to obtain a decomposition
of $G^{\diamond}$ into edge-disjoint spanning subgraphs $H(i,i')$ and $H''(i,i')$ (for all $1\le i,i' \le K$) which satisfy the following
properties, where $G'(i,i'):=H(i,i')+H''(i,i')$:
\begin{itemize}
\item[\rm (b$'_1$)] Each $H(i,i')$ contains only $A_0A_i$-edges and $B_0B_{i'}$-edges.
\item[\rm (b$'_2$)] $H''(i,i')\subseteq G^{\diamond}[A',B']$.
Moreover, all but at most  $20 \eps n/K^2$ edges of $H''(i,i')$ lie in $G^{\diamond}[A_0 \cup A_i, B_0 \cup B_{i'}]$.
\item[\rm (b$'_3$)] $e(H''(i,i')) =  2 \left\lceil e_{G^{\diamond}}(A',B') / (2 K^2) \right\rceil$ or $e(H''(i,i')) =  2 \left\lfloor e_{G^{\diamond}}(A',B') / (2 K^2) \right\rfloor$.
In particular, $2\alpha n \le e(H''(i,i')) \le 19 \alpha n /5$ by~\eqref{eW'}.
\item[\rm (b$'_4$)] $d_{H''(i,i')}(v )  =  ( d_{G^{\diamond}[A',B']}(v)  \pm 25 \eps n)/K^2$ for all $v \in V_0$.
\item[\rm (b$'_5$)] $d_{G'(i,i')}(v )  =  ( d_{G^\diamond}(v)  \pm 25 \eps n)/K^2=\left( 2 \alpha  \pm   25 \eps/K^2 \right) n$ for all $v \in V_0$ by~\eqref{eq:degvV0}.
\item[\rm (b$'_6$)] Each $w \in W_0$ satisfies $d_{H''(i,i')}(w) \le \lceil d_{G{^{\diamond}}[A',B']}(w) / K^2 \rceil \le \alpha n $ by~\eqref{degW0}.
\end{itemize}
Our next aim is to apply Lemma~\ref{lma:BESdecomprelim2} to each $H''(i,i')$ to obtain suitable Hamilton exceptional system candidates
(in particular almost all of them will be `localized').
So consider any $ 1 \le i,i' \le K$%
\COMMENT{AL: added $1 \le$}
 and let $H'':=H''(i,i')$. We claim that there exists $0 \le \eta \le  9/10$
such that $H''$ satisfies the following conditions (which in turn imply conditions (c$_1$)--(c$_5$) of Lemma~\ref{lma:BESdecomprelim2}): 
\begin{itemize}
	\item[\rm (c$_1'$)] $e(H'') = 2(1+ \eta)\alpha n$ and $\eta \alpha n\in \mathbb{N}$.
	\item[\rm (c$_2'$)] $e(H''- W' ) \le 199 \alpha n/100$ and $d_{H''}(w) \ge 13 \alpha n / 25 $ for all $w \in W'$.
	\item[\rm (c$_3'$)] $d_{H''}(w) \le \alpha n$ for all $w \in W_0$ and $d_{H''}(w') \le 41 \alpha n /40$ for all $w'\in W'\setminus W_0$.%
	\COMMENT{Recall that if $W_0\neq W'$ then $W_0=\{w_,w_2\}$ and $W'=\{w_,w_2,w_3\}$.}
	\item[\rm (c$_4'$)] For all $w \in W'$ and all $v \in V(G) \setminus W'$ we have $d_{H''}(w) - d_{H''}(v) \ge \alpha n/150$.
	\item[\rm (c$_5'$)] For all $v \in A \cup B$ we have $d_{H''}(v) \le \epszero n$.
\end{itemize}
Clearly, (b$_3'$) implies the first part of (c$_1'$). Since $e(H'')$ is even by~(b$_3'$) and $\alpha n\in \mathbb{N}$,
it follows that $\eta \alpha n\in \mathbb{N}$.
To verify the first part of (c$_2'$), note that (b$'_3$) and (b$'_4$) together imply that
\begin{align*}
e(H''- W' ) & = e(H'')-\sum_{w\in W'} d_{H''}(w)+e(H''[W']) \\ & \le 
2 \left\lceil e_{G^{\diamond}}(A',B') / (2 K^2) \right\rceil- \sum_{w\in W'} (d_{G^\diamond[A',B']}(w) - 25 \eps n)/K^2 +3\\
& \le (e_{G^\diamond-W'}(A',B') +80\eps n)/K^2.
\end{align*}
Together with Lemma~\ref{critical}(iv) this implies that
$$e(H''- W' )\le (e_{G-W'}(A',B') +80\eps n)/K^2\le ((3D/4+5)+80\eps n)/K^2\le 199 \alpha n/100.$$
To verify the second part of (c$_2'$), note that by (\ref{eq:degrees}) and Lemma~\ref{lma:BESdecomcritical}(iii) each $w\in W'$
satisfies $d_{G^\diamond [A',B']}(w)\ge d_{G[A',B']}(w)-\phi n \ge 21D/80-\phi n$. Together with (b$_4'$)
this implies $d_{H''}(w) \ge 26\alpha n/50$. Thus (c$_2'$) holds. By (b$_6'$) we have $d_{H''}(w) \le \alpha n $ for all $w \in W_0$.
If $w' \in W' \setminus W_0$, then Lemma~\ref{lma:BESdecomcritical}(ii) implies $d_{G[A',B']}(w') \le D/2 \le 51 K^2 \alpha n /50$.
Thus, $d_{H''}(w') \le 41 \alpha n /40$ by~(b$_4'$). Altogether this shows that~(c$_3'$) holds.
(c$_4'$) follows from (\ref{eq:degrees}), (b$'_4$) and the fact that $d_{G^\diamond [A',B']}(v)\ge d_{G[A',B']}(v)-\phi n$ for all
$v\in V(G)$ by Lemma~\ref{lma:BESdecomcritical}(iii).
(c$_5'$) holds since $d_{H''}(v) \le d_{G^\diamond[A',B']}(v) \le \eps_0 n$ for all $v\in A\cup B$ by (ESch3).

Now we apply Lemma~\ref{lma:BESdecomprelim2} in order to decompose $H''$ into $\alpha n$ edge-disjoint Hamilton exceptional system candidates $F_1,\dots,F_{\alpha n}$ such that $e(F_s) \in \{2,4\}$
for all $s\le \alpha n$ and such that at least $\alpha n /200$ of $F_s$ satisfy $e(F_s)=2$ and $d_{F_s}(w)=1$ for all $w \in W_0$. 
Let
\begin{align*}
\gamma & := \alpha - \frac{\lambda}{K^2} & 
& \text{and} &
\gamma' & : = \frac{\lambda}{K^2}.
\end{align*}   
Recall that by (b$_2'$) all but at most $20 \eps n /K^2 \le \eps' n $ edges of $H''$ lie in $G^{\diamond}[A_0 \cup A_i, B_0 \cup B_{i'}]$.
Together with (\ref{alphahier3}) this ensures that we can relabel the $F_s$ if necessary to obtain $\alpha n $ edge-disjoint Hamilton exceptional system
candidates
$F_1(i,i'), \dots , F_{\gamma n}(i,i')$ and $F'_1(i,i'), \dots , F'_{\gamma ' n}(i,i')$ such that the following properties hold:
\begin{itemize}
	\item[(a$'$)] $F_s(i,i')$ is an $(i,i')$-HESC for every $s \le \gamma n$.
	Moreover, $\gamma ' n$ of the $F_s(i,i')$ satisfy $e(F_s(i,i')) = 2$ and $d_{F_s(i,i')}(w) = 1$ for all $w \in W_0$.%
	\item[(b$'$)] $e(F'_s(i,i')) = 2$ for all but at most $\eps' n$ of the $F'_s(i,i')$.
	\item[(c$'$)] $e(F_s(i,i')), e(F'_s(i,i'))\in \{2,4\}$.
\end{itemize}
For (b$'$) and the `moreover' part of (a$'$), we use  that $ \alpha n /200 - \eps' n \ge 2\lambda n/K^2 = 2\gamma'n$.
Our next aim is to apply Lemma~\ref{BEScons} with $G^\diamond$ playing the role of $G^*$ to extend the above exceptional system candidates into 
exceptional systems.
Clearly conditions (i) and (ii) of Lemma~\ref{BEScons} hold. (iii) follows from (b$'_1$). (iv) and (v) follow from (a$'$)--(c$'$).
(vi) follows from Lemma~\ref{lma:BESdecomcritical}(i),(iii). Finally, (vii) follows from (b$'_5$) since $G'(i,i')$ plays the role of $G^*(i,i')$.
Thus we can indeed apply Lemma~\ref{BEScons} to obtain a decomposition of $G^\diamond$ into $K^2\alpha n$
edge-disjoint Hamilton exceptional systems $J_1(i,i'),\dots,J_{\gamma n}(i,i')$ and $J'_1(i,i'),\dots,J'_{\gamma' n}(i,i')$
with parameter $\eps_0$, where $1\le i,i'\le K$, such that $J_s(i,i')$ is an $(i,i')$-HES which is a faithful extension of  $F_s(i,i')$ for all $s \le \gamma n$
and $J'_s(i,i')$ is a faithful extension of $F'_s(i,i')$ for all $s \le \gamma' n$.
Then the set $\mathcal{J}$ of all these exceptional systems is as required in Lemma~\ref{lma:BESdecomcritical}.
(Since $W_0$ contains $\{w_1,w_2\} \cap W$, the `moreover part' of (a$'$) implies the `moreover part' of Lemma~\ref{lma:BESdecomcritical}(b).)
\endproof

%%%%%%%%%%%%%%%%%%%%%%%%%%%%%%%%%%%%%%%%%%%%%%%%%%%%%%%%%%%%%%%%%%%%%%%%%%%%%%%%%%%%%%

\subsection{The case when $e_G(A',B') < D$}\label{1factsec}

The aim of this section is to prove the following analogue of Lemma~\ref{lma:BESdecom} for the case when $e_G(A',B')< D$.
In this case, we do not need to prove any auxiliary lemmas first, as we can apply those proved in the other two cases
(Lemmas~\ref{BEScons} and~\ref{lma:movecritical}).

Recall that Proposition~\ref{prp:e(A',B')} implies that in the current case we have $n = 0 \pmod4$, $D = n/2-1$ and $|A'| = |B'| = n/2$.

\begin{lemma} \label{lma:PBESdecom}
Suppose that $0 <  1/n  \ll \epszero \ll  \eps \ll   \lambda, 1/K \ll 1$, that $0 \le \phi  \ll 1$ and that
$n/4, K, m, \lambda n/K^2, (n/2-1 - \phi n)/(2K^2) \in \mathbb{N}$.%
    \COMMENT{Previously also had that $\phi n \in \mathbb{N}$. But this follows since $n/2-1,K, (n/2-1 - \phi n)/(2K^2) \in \mathbb{N}$.}
Suppose that the following conditions hold:
\begin{itemize}
	\item[\rm (i)] $G$ is an $(n/2 -1)$-regular graph on $n$ vertices.
	\item[\rm (ii)] $\mathcal{P}$ is a $(K, m, \epszero)$-partition of $V(G)$ such that
$\Delta(G[A',B']) \le n/4$ and $|A'| = |B'| = n/2$.%
\COMMENT{Probably could actually assume that $\Delta(G[A',B']) \le n/4-1$. We also have $e_G(A',B') < D$ but this statement
is not needed explicitly. This statement is incorporated in (v)}
	\item[\rm (iii)] $G_0$ is a subgraph of $G$ such that $G[A_0]+G[B_0] \subseteq G_0$ and $d_{G_0}(v) = \phi n $ for all $v \in V_0$.
	\item[\rm (iv)] Let $G^{\diamond} := G - G[A]  - G[B] - G_0$. $e_{G^\diamond} (A',B')$ is even and
$(G^{\diamond}, \mathcal{P})$ is a $(K, m, \epszero,\eps)$-exceptional scheme.
	\item[\rm (v)] $\Delta(G^{\diamond}[A',B']) \le e_{G^{\diamond}}(A',B')/2  \le (n/2 - 1-  \phi n ) /2$.
\end{itemize}
Then there exists a set $\mathcal{J}$ consisting of $(n/2-1-\phi n)/2$ edge-disjoint exceptional systems in $G^{\diamond}$
which satisfies the following properties:
\begin{itemize}
    \item[\rm (a)] Together the exceptional systems in $\mathcal{J}$ cover all edges of $G^\diamond$.
Each $J_s$ in $\mathcal{J}$%
    \COMMENT{Daniela: added in $\mathcal{J}$}
is either a Hamilton exceptional system with $e_{J_s}(A',B') = 2$ or a matching exceptional system.
	\item[\rm (b)] For all $1\le i,i' \le K$, the set $\mathcal{J}$ contains $(n/2-1-(\phi n+2 \lambda))/(2K^2)$ $(i,i')$-ES.
\end{itemize}
\end{lemma}
As in the other two cases, in~\cite{paper1} we will use some of the exceptional systems in (b) to construct the robustly decomposable graph $G^{\rm rob}$.
Unlike the critical case with $e_G(A',B') \ge D$, there is no need to `track' the degrees of the vertices $w_i$ of high degree in $G[A',B']$
this time (this is due to the very special structure of the exceptional systems produced in this case). 

\begin{proof}
Let $\eps'$ be a new constant such that $\eps \ll \eps' \ll \lambda, 1/K$ and
set 
\begin{equation} \label{alphaD}
2\alpha n:= \frac{n/2 - 1 - \phi n }{K^2}.
\end{equation}
Similarly as in the proof of Lemma~\ref{lma:BESdecomcritical} we have
\begin{equation}\label{alphahier4}
\eps \ll \eps'  \ll \lambda, 1/K, \alpha \ll 1.
\end{equation}
We claim that $G^{\diamond}$ can be decomposed into edge-disjoint spanning subgraphs $H(i,i')$ and $H''(i,i')$ (for all $1\le i,i' \le K$)
which satisfy the following properties, where $G'(i,i'):=H(i,i')+H''(i,i')$:
\begin{itemize}
\item[\rm (b$_1'$)] Each $H(i,i')$ contains only $A_0A_i$-edges and $B_0B_{i'}$-edges.
 \item[\rm (b$_2'$)] $H''(i,i')\subseteq G^{\diamond}[A',B']$.
Moreover, all but at most  $\eps' n$ edges of $H''(i,i')$ lie in $G^{\diamond}[A_0 \cup A_i, B_0 \cup B_{i'}]$.
\item[\rm (b$_3'$)] $e ( H''(i,i') )$ is even and $e(H''(i,i')) \le 2 \alpha n $.
\item[\rm (b$_4'$)] $\Delta(H''(i,i')) \le e ( H''(i,i') )/2$.
\item[\rm  (b$_5'$)] $d_{G'(i,i')}(v )  =  (2\alpha   \pm  \eps' )n$ for all $v\in V_0$.
\end{itemize}
To see this, let us first consider the case when $ e_{G^{\diamond}}(A',B') \le 300 \eps n $.
Apply Lemma~\ref{lma:randomslice} to $G^{\diamond}$ in order to obtain a decomposition
of $G^{\diamond}$ into edge-disjoint spanning subgraphs $H(i,i')$ and $H'(i,i')$ (for all $1\le i,i' \le K$) which satisfy
Lemma~\ref{lma:randomslice}(a$_1$)--(a$_5$). Set $H''(1,1) := \bigcup_{i,i' \le K} H'(i,i') = G^{\diamond}[A',B']$ and
$H''(i,i') := \emptyset$ for all other pairs $1 \le i,i' \le K$.%
\COMMENT{AL: added $1 \le$}
 Then (b$_1'$) follows from (a$_1$).
(b$_2'$) follows from our definition of the $H''(i,i')$ and our assumption that $ e_{G^{\diamond}}(A',B') \le 300 \eps n< \eps' n <\alpha n$.
Together with Lemma~\ref{lma:PBESdecom}(iv) this also implies (b$_3'$). (b$_4'$) follows from Lemma~\ref{lma:PBESdecom}(v).
Note that by Lemma~\ref{lma:PBESdecom}(i) and~(iii),
every $v\in V_0$ satisfies $d_{G^\diamond}(v)=n/2-1-\phi n=2K^2\alpha n$.
So, writing $G(i,i'):=H(i,i')+H'(i,i')$, (a$_5$) implies that
$$
d_{G'(i,i')}(v )=d_{G(i,i')}(v)\pm 300 \eps n =(2\alpha \pm 4\eps/K^2)n\pm 300 \eps n=(2\alpha   \pm  \eps' )n.
$$
Thus (b$_5'$) holds too.

So let us next consider the case when $ e_{G^{\diamond}}(A',B') > 300 \eps n $.
Let $W_0$ be the set of all those vertices $v \in V(G)$ for which $d_{G^{\diamond}[A',B']}(v) \ge 3e_{G^{\diamond}}(A',B')/8$.
Then clearly $|W_0| \le 2$. Moreover, each $v \in V(G) \setminus W_0$ satisfies
\begin{align}\label{eq:degnonW0}
d_{G^{\diamond}[A',B']}(v) + 26 \eps n < 3e_{G^{\diamond}}(A',B')/8 + e_{G^{\diamond}}(A',B')/8 = e_{G^{\diamond}}(A',B')/2.
\end{align}
Recall from Lemma~\ref{lma:PBESdecom}(v) that each $w\in W_0$ satisfies
$d_{G^{\diamond}[A',B']}(w)\le e_{G^{\diamond}}(A',B')/2$. So we can
apply Lemma~\ref{lma:movecritical} to $G^{\diamond}$ in order to obtain a decomposition
of $G^{\diamond}$ into edge-disjoint spanning subgraphs $H(i,i')$ and $H''(i,i')$ (for all $1\le i,i' \le K$) which satisfy
Lemma~\ref{lma:movecritical}(b$_1$)--(b$_7$). Then (b$_1$) and (b$_2$) imply (b$'_1$) and (b$'_2$).
(b$_3'$) follows from (b$_3$),~(\ref{alphaD}) and Lemma~\ref{lma:PBESdecom}(v). Note that (b$_3$), (b$_4$) and~(\ref{eq:degnonW0})
together imply that
\begin{equation}\label{eq:degH''v}
d_{H''(i,i')}(v) \le \frac{e_{G^{\diamond}}(A',B')/2-\eps n}{K^2}\le \frac{e ( H''(i,i') )}{2}
\end{equation}
for all $v \in V_0 \setminus W_0$. Note that each $v\in A\cup B$ satisfies $d_{H''(i,i')}(v) \le d_{G^\diamond[A',B']}(v)\le \eps_0 n$ by
Lemma~\ref{lma:PBESdecom}(iv) and (ESch3). Together with the fact that $e ( H''(i,i'))\ge 2\lfloor 300\eps n/(2K^2)\rfloor \ge 2\eps_0 n$
by~(b$_3$), this implies that~(\ref{eq:degH''v}) also holds for all $v\in A\cup B$. 
 Together with (b$_7$) this implies (b$'_4$).
(b$'_5$) follows from (b$_5$) and the fact that by Lemma~\ref{lma:PBESdecom}(i) and~(iii)
every $v\in V_0$ satisfies $d_{G^\diamond}(v)=n/2-1-\phi n=2K^2\alpha n$.
So (b$'_1$)--(b$'_5$) hold in all cases.

We now decompose the localized subgraphs $H''(i,i')$ into exceptional system candidates. For this, 
fix $i,i' \le K$ and write $H''$ for $H''(i,i')$. By (b$_4'$) we have $\Delta(H'') \le e ( H'')/2$ and so $\chi'(H'') \le e ( H'')/2$.
Apply Proposition~\ref{prop:matchingdecomposition} with $e ( H'')/2$ playing the role of $m$ to decompose $H''$ into $e ( H'')/2$ edge-disjoint
matchings, each of size~$2$. Note that $\alpha n-e ( H'')/2\ge 0$ by (b$_3'$). So we can add some empty matchings to obtain a
decomposition of $H''$ into $\alpha n$ edge-disjoint $M_1, \dots ,M_{\alpha n}$ such that each $M_s$ is either empty or has size~2.
Let
\begin{align*}
\gamma & := \alpha - \frac{\lambda}{K^2} & 
& \text{and} &
\gamma' & : = \frac{\lambda}{K^2}.
\end{align*} 
Recall from (b$_2'$) that all but at most $\eps' n \le \gamma' n$ edges of $H''$ lie in $G^{\diamond}[A_0 \cup A_i, B_0 \cup B_{i'}]$.
Hence by relabeling if necessary, we may assume that  $M_s \subseteq G^{\diamond}[A_0 \cup A_i ,  B_0 \cup B_{i'}]$ for every $s \le \gamma n$.
So by setting $F_s(i,i') := M_s$ for all $s \le \gamma n$ and $F_s'(i,i') := M_{ \gamma n + s }$ for all $s \le \gamma' n$
we obtain a decomposition of $H''$ into edge-disjoint exceptional system candidates $F_1(i,i'), \dots , F_{\gamma n}(i,i')$
and $ F'_1(i,i'), \dots , F'_{\gamma ' n}(i,i')$ such that
the following properties hold:
\begin{itemize}
	\item[(a$'$)] $F_s(i,i')$ is an $(i,i')$-ESC for every $s \le \gamma n$.
	\item[(b$'$)] Each $F_s(i,i')$ is either a matching exceptional system candidate with $e(F_s(i,i')) = 0$ or a
Hamilton exceptional system candidate with $e(F_s(i,i')) = 2$. The analogue holds for each $F'_{s'}(i,i')$.
\end{itemize}
Our next aim is to apply Lemma~\ref{BEScons} with $G^\diamond$ playing the role of $G^*$, to extend the above 
exceptional system candidates into exceptional systems.
Clearly conditions (i) and (ii) of Lemma~\ref{BEScons} hold. (iii) follows from (b$'_1$). (iv) and (v) follow from (a$'$) and (b$'$).
(vi) follows from Lemma~\ref{lma:PBESdecom}(i),(iii). Finally, (vii) follows from (b$'_5$) since $G'(i,i')$ plays the role of $G^*(i,i')$ in Lemma~\ref{BEScons}.
Thus we can indeed apply Lemma~\ref{BEScons} to obtain a decomposition of $G^\diamond$ into $K^2\alpha n$
edge-disjoint exceptional systems $J_1(i,i'),\dots,J_{\gamma n}(i,i')$ and $J'_1(i,i'),\dots,J'_{\gamma' n}(i,i')$,
where $1\le i,i'\le K$, such that $J_s(i,i')$ is an $(i,i')$-ES which is a faithful extension of  $F_s(i,i')$ for all $s \le \gamma n$
and $J'_s(i,i')$ is a faithful extension of $F'_s(i,i')$ for all $s \le \gamma' n$.
Then the set $\mathcal{J}$ of all these exceptional systems is as required in Lemma~\ref{lma:PBESdecom}.
\end{proof}

%%%%%%%%%%%%%%%%%%%%%%%%%%%%%%%%%%%%%%%%%%%%%%%%%%%%%%%%%%%%%%%%%%%%%%%%%%%%%%%%

\section{Acknowledgements}

We are grateful to B\'ela Csaba and Andrew Treglown for helpful discussions.

\medskip

{\footnotesize \obeylines \parindent=0pt

Daniela K\"{u}hn, Allan Lo, Deryk Osthus 
School of Mathematics
University of Birmingham
Edgbaston
Birmingham
B15 2TT
UK
}
\begin{flushleft}
{\it{E-mail addresses}:
\tt{\{d.kuhn,s.a.lo,d.osthus\}@bham.ac.uk}}
\end{flushleft}


\begin{thebibliography}{10}
\bibitem{1factorization} A.G.~Chetwynd and A.J.W.~Hilton, 
Regular graphs of high degree are 1-factorizable, 
\emph{Proc. London Math. Soc.} {\bf 50} (1985), 193--206.

\bibitem{CH} A.G.~Chetwynd  and A.J.W. Hilton, 
1-factorizing regular graphs of high degree---an improved bound, 
\emph{Discrete Math.} {\bf 75} (1989), 103--112.


\bibitem{paper2} B.~Csaba, D.~K\"uhn, A.~Lo, D.~Osthus and A.~Treglown, Proof of the $1$-factorization and Hamilton decomposition conjectures II: the bipartite case, preprint.

\bibitem{paper3} B.~Csaba, D.~K\"uhn, A.~Lo, D.~Osthus and A.~Treglown, Proof of the $1$-factorization and Hamilton decomposition conjectures III: approximate decompositions, preprint.

%\bibitem{JF} J.~C.~Fournier, Colorations des ar\^etes d'un graphe, \emph{Cahiers du Centre d'\'Etudes de Recherche Op\'erationnelle}~{\bf 15} (1973) %311--314.





\bibitem{Janson&Luczak&Rucinski00} S.~Janson, T.~{\L}uczak\ and\ A.~Ruci\'nski,
{\em Random Graphs}, Wiley, 2000.




\bibitem{KLOmindeg} D.~K\"uhn, J.~Lapinskas and D.~Osthus, Optimal packings of Hamilton cycles in graphs of high minimum degree,
\emph{ Combin. Probab. Comput.}~\textbf{22} (2013), 394--416.


\bibitem{paper1} D.~K\"uhn, A.~Lo, D.~Osthus and A.~Treglown, Proof of the $1$-factorization and Hamilton decomposition conjectures I:
the two cliques case, preprint.

\bibitem{Kelly} D.~K\"uhn and D.~Osthus, Hamilton decompositions of regular expanders: a proof of Kelly's conjecture for large tournaments,
\emph{Advances in Mathematics}~\textbf{237} (2013), 62--146.


\bibitem{initconj} C.St.J.A.~Nash-Williams, Hamiltonian lines
in graphs whose vertices have sufficiently large valencies, in \emph{Combinatorial
theory and its applications, III (Proc. Colloq., Balatonf\"ured, 1969)},
North-Holland, Amsterdam (1970), 813--819.


\bibitem{decompconj} C.St.J.A.~Nash-Williams, Hamiltonian
arcs and circuits, in \emph{Recent Trends in Graph Theory (Proc. Conf.,
New York, 1970)}, Springer, Berlin (1971), 197--210.


\bibitem{west} D.B.~West, 
Introduction to Graph Theory (2nd Edition), Pearson 2000.

\end{thebibliography}
\end{document}